\documentclass[10pt,amssymb,amsfonts,psfig]{amsart}
\begin{document}

\newtheorem{thm}{Theorem}[section]
\newtheorem{cor}[thm]{Corollary}
\newtheorem{lem}[thm]{Lemma}
\newtheorem{prop}[thm]{Proposition}
\newtheorem{obs}[thm]{Observation}
\theoremstyle{remark}
\newtheorem{rem}{Remark}[section]
\newtheorem{notation}{Notation}
\newtheorem{exc}[thm]{Exercise}
\newtheorem{exe}[thm]{Exercise}
\newtheorem{problem}[thm]{Problem}

\theoremstyle{definition}
\newtheorem{defn}{Definition}[section]

\theoremstyle{plain}
\newtheorem*{thmA}{Main Theorem (physical version)}
\newtheorem*{thmB}{Main Theorem (dynamical version)}
\newtheorem*{thmC}{Theorem (Global Lee-Yang-Fisher Current)}
\newtheorem*{thmRIG}{Local Rigidity Theorem}
\newtheorem*{LY Theorem}{Lee-Yang Theorem}
\newtheorem*{LY Theorem2}{General Lee-Yang Theorem}
\newtheorem*{LY TheoremBC}{Lee-Yang Theorem with Boundary Conditions}
\newtheorem*{LY Theorem2BC}{General Lee-Yang Theorem with Boundary Conditions}
\newtheorem*{conj}{Conjecture}

\numberwithin{equation}{section}
\numberwithin{figure}{section}

\newcommand{\figref}[1]{Figure~\ref{#1}}
\newcommand{\Rmig}{R} 
\newcommand{\Cmig}{C} 
\newcommand{\Cmigbl}{C_0}
\newcommand{\TOPmig}{\mathrm{T}} 
\newcommand{\BOTTOMmig}{\mathrm{B}} 
\newcommand{\FIXmig}{b} 
\newcommand{\CFIXmig}{e} 
\newcommand{\INDmig}{a} 
\newcommand{\Smig}{S} 

\newcommand{\thin} {{\mathrm{thin}}}
\newcommand{\thick} {{ \mathrm{thick}}}
\newcommand{\loc}{{\mathrm {loc}}}
\newcommand{\hor}{{\mathrm {hor}}}
\newcommand{\ver}{{\mathrm {ver}}}
\newcommand{\ess}{{\mathrm{ess}}}
\newcommand{\ness}{{\mathrm{ne}}}
\newcommand{\pro}{{\mathrm{pr}}}

\newcommand{\locflux} {\WW^\perp_\loc}

\newcommand{\re}{{\mathrm{(r)}}}
\newcommand{\essre}{{\mathrm{ess(r)}}}
\newcommand{\nere}{{\mathrm{ne(r)}}}
\newcommand{\verre}{{\mathrm{ver(r)}}}

\newcommand{\Cyl}{{\mathrm{Cyl}}}
\newcommand{\hsubset}{{ \,\underset{h}{\subset}\, }    }
\newcommand{\hsupset}{{ \, \underset{h}{\supset} \, }     }

\newcommand{\lp}{{\,\mathrm {dominates}\,}}

\newcommand{\Rphys}{\mathcal R} 
\newcommand{\Cphys}{\mathcal C} 
\newcommand{\Solid}{\mathcal{SC}}
\newcommand{\Solidmig}{SC}
\newcommand{\Secphys}{\Pi}
\newcommand{\Secmig}{P}
\newcommand{\TOPphys}{\mathcal T} 
\newcommand{\Cphystl}{{\mathcal C}_1}
\newcommand{\Cphysbl}{{\mathcal C}_0}
\newcommand{\Cphyslow}{{\mathcal C}_*}
\newcommand{\BOTTOMphys}{\mathcal B} 
\newcommand{\FIXphys}{\beta} 
\newcommand{\CFIXphys}{\eta} 
\newcommand{\INDphys}{\alpha} 
\newcommand{\Sphys}{\mathcal S} 
\newcommand{\Icurve}{\mathcal G} 
\newcommand{\Imig}{G} 
\newcommand{\PI}{{\mathcal A}}
\newcommand{\crittemps}{{\mathscr C}} 
\newcommand{\epoints}{{\mathscr E}} 

\newcommand{\Line}{L}
\newcommand{\Lzero}{L_0}
\newcommand{\Lone}{L_1}
\newcommand{\Ltwo}{L_2}
\newcommand{\Lthree}{L_3}
\newcommand{\Lfour}{L_4}

\newcommand{\LLzero}{\LL_0}
\newcommand{\LLone}{\LL_1}

\newcommand{\Div}{{\rm Div}}
\newcommand{\Tongue}{\Upsilon}

\newcommand{\Par}{{{\mathcal P}}}
\newcommand{\inv}{{\iota}}

\newcommand{\leg}{{\mathrm {leg}}}
\newcommand{\il}{{\mathrm {il}}}

\newcommand{\Mconv}{{\rm M}}
\newcommand{\Hconv}{{\rm H}}
\newcommand{\tconv}{{\rm t}}

\newcommand{\Weight}{W}

\newcommand{\KSQRT}{\widehat \KK}

\newcommand{\ex}{{\mathrm{exc}}}
\newcommand{\out}{{\mathrm{out}}}

\newcommand{\vertbondpp}{\, \underset{\oplus}{\overset{\oplus}{|}} \, }
\newcommand{\vertbondpm}{\, \underset{\ominus}{\overset{\oplus}{|}} \,}
\newcommand{\vertbondmp}{\, \underset{\oplus}{\overset{\ominus}{|}} \, }
\newcommand{\vertbondmm}{\, \underset{\ominus}{\overset{\ominus}{|}} \, }

\font\nt=cmr7

\def\note#1
{\marginpar
{\nt $\leftarrow$
\par
\hfuzz=20pt \hbadness=9000 \hyphenpenalty=-100 \exhyphenpenalty=-100
\pretolerance=-1 \tolerance=9999 \doublehyphendemerits=-100000
\finalhyphendemerits=-100000 \baselineskip=6pt
#1}\hfuzz=1pt}

\def\note#1{}

 \newcommand{\bignote}[1]{\begin{quote} \sf #1 \end{quote}}

 \long\def\bignote#1{}


\newcommand{\lo}{{\mathrm{lo}}}
\newcommand{\lobr}{{\mathrm{lo-br}}}
\newcommand{\sh}{{\mathrm{sh}}}
\newcommand{\sei}{{\mathrm{si}}}
\newcommand{\pe}{{\mathrm{pe}}}

\newcommand{\ol}{\overline}
\newcommand{\ul}{\underline}

\newcommand{\pp}{{\mathfrak{p}}}
\newcommand{\qq}{{\mathfrak{q}}}
\newcommand{\bb}{{\mathfrak{b}}}

\newcommand{\Kfilled}{{\mathcal {K} }}
\newcommand{\Jul}{{\mathcal{J}}}
\newcommand{\Sec} {{{S}}}
\newcommand{\Zec} {{{Z}}}
\newcommand{\Hub}{\TT}



\newcommand{\QED}{\rlap{$\sqcup$}$\sqcap$\smallskip}

\def\sss{\subsubsection}

\newcommand{\correspond}{\Psi}
\newcommand{\conjugacy}{\psi}

\newcommand{\rank}{\rm rank}
\newcommand{\di}{\partial}
\newcommand{\dibar}{\bar\partial}
\newcommand{\hookra}{\hookrightarrow}
\newcommand{\ra}{\rightarrow}
\newcommand{\hra}{\hookrightarrow}
\newcommand{\imply}{\Rightarrow}
\def\lra{\longrightarrow}
\newcommand{\wc}{\underset{w}{\to}}
\newcommand{\tu}{\textup}

\def\ssk{\smallskip}
\def\msk{\medskip}
\def\bsk{\bigskip}
\def\noi{\noindent}
\def\nin{\noindent}
\def\lqq{\lq\lq}
\def\sm{\setminus}
\def\bolshe{\succ}
\def\ssm{\smallsetminus}
\def\tr{{\text{tr}}}
\def\Crit{{\mathrm{Crit}}}

\newcommand{\ctg}{\operatorname{ctg}}
\newcommand{\diam}{\operatorname{diam}}
\newcommand{\dist}{\operatorname{dist}}
\newcommand{\Hdist}{\operatorname{H-dist}}
\newcommand{\cl}{\operatorname{cl}}
\newcommand{\inter}{\operatorname{int}}
\renewcommand{\mod}{\operatorname{mod}}
\newcommand{\card}{\operatorname{card}}
\newcommand{\tl}{\tilde}
\newcommand{\ind}{ \operatorname{ind} }
\newcommand{\Dist}{\operatorname{Dist}}
\newcommand{\Graph}{\operatorname{Graph}}
\newcommand{\len}{\operatorname{\l}}
\newcommand{\vol}{\operatorname{vol}}

\renewcommand{\Re}{\operatorname{Re}}
\renewcommand{\Im}{\operatorname{Im}}

\newcommand{\orb}{\operatorname{orb}}
\newcommand{\HD}{\operatorname{HD}}
\newcommand{\supp}{\operatorname{supp}}
\newcommand{\id}{\operatorname{id}}
\newcommand{\length}{\operatorname{length}}
\newcommand{\dens}{\operatorname{dens}}
\newcommand{\meas}{\operatorname{meas}}
\newcommand{\area}{\operatorname{area}}
\newcommand{\per}{\operatorname{per}}
\renewcommand{\Im}{\operatorname{Im}}

\renewcommand{\d}{{\diamond}}

\newcommand{\lef}{{\mathrm{left}}}
\newcommand{\righ}{{\mathrm{right}}}

\newcommand{\Dil}{\operatorname{Dil}}
\newcommand{\Ker}{\operatorname{Ker}}
\newcommand{\tg}{\operatorname{tg}}
\newcommand{\codim}{\operatorname{codim}}
\newcommand{\isom}{\approx}
\newcommand{\comp}{\circ}
\newcommand{\esssup}{\operatorname{ess-sup}}
\newcommand{\Rat}{{\mathrm{Rat}}}
\newcommand{\hot}{{\mathrm{h.o.t.}}}
\newcommand{\Conf}{{\mathrm{Conf}}}

\newcommand{\SLa}{\underset{\La}{\Subset}}

\newcommand{\const}{\mathrm{const}}
\def\loc{{\mathrm{loc}}}
\def\fib{{\mathrm{fib}}}
\def \br{{\mathrm{br}}}

\newcommand{\eps}{{\epsilon}}
\newcommand{\epsi}{{\epsilon}}
\newcommand{\veps}{{\varepsilon}}

\newcommand{\Ga}{{\Gamma}}
\newcommand{\De}{{\Delta}}
\newcommand{\de}{{\delta}}
\newcommand{\la}{{\lambda}}
\newcommand{\La}{{\Lambda}}
\newcommand{\si}{{\sigma}}
\newcommand{\Si}{{\Sigma}}
\newcommand{\Om}{{\Omega}}
\newcommand{\om}{{\omega}}
\newcommand{\Ups}{{\Upsilon}}

\newcommand{\al}{{\alpha}}
\newcommand{\ba}{{\mbox{\boldmath$\alpha$} }}
\newcommand{\be}{{\beta}}
\newcommand{\bbe}{{\mbox{\boldmath$\beta$} }}
\newcommand{\bk}{{\boldsymbol{\kappa}}}
\newcommand{\bg}{{\boldsymbol{\gamma}}}

\newcommand{\bare}{{\bar\eps}}

\newcommand{\Ray}{{\mathcal R}}
\newcommand{\Eq}{{\mathcal E}}
\newcommand{\PR}{PR}

\newcommand{\AAA}{{\mathcal A}}
\newcommand{\BB}{{\mathcal B}}
\newcommand{\CC}{{\mathcal C}}
\newcommand{\DD}{{\mathcal D}}
\newcommand{\EE}{{\mathcal E}}
\newcommand{\EEE}{{\mathcal O}}
\newcommand{\II}{{\mathcal I}}
\newcommand{\FF}{{\mathcal F}}
\newcommand{\GG}{{\mathcal G}}
\newcommand{\JJ}{{\mathcal J}}
\newcommand{\HH}{{\mathcal H}}
\newcommand{\KK}{{\mathcal K}}
\newcommand{\LL}{{\mathcal L}}
\newcommand{\MM}{{\mathcal M}}
\newcommand{\NN}{{\mathcal N}}
\newcommand{\OO}{{\mathcal O}}
\newcommand{\PP}{{\mathcal P}}
\newcommand{\QQ}{{\mathcal Q}}
\newcommand{\QM}{{\mathcal QM}}
\newcommand{\QP}{{\mathcal QP}}
\newcommand{\QL}{{\mathcal Q}}

\newcommand{\RR}{{\mathcal R}}
\newcommand{\SSS}{{\mathcal S}}
\newcommand{\TT}{{\mathcal T}}
\newcommand{\TTT}{{\mathcal P}}
\newcommand{\UU}{{\mathcal U}}
\newcommand{\VV}{{\mathcal V}}
\newcommand{\WW}{{\mathcal W}}
\newcommand{\XX}{{\mathcal X}}
\newcommand{\YY}{{\mathcal Y}}
\newcommand{\ZZ}{{\mathcal Z}}

\newcommand{\AS}{{\mathcal{AS}}}
\newcommand{\SAS}{{\mathcal{SAS}}}

\newcommand{\A}{{\Bbb A}}
\newcommand{\BBB}{{\Bbb B}}
\newcommand{\C}{{\Bbb C}}
\newcommand{\bC}{{\bar{\Bbb C}}}
\newcommand{\D}{{\Bbb D}}
\newcommand{\Hyp}{{\Bbb H}}
\newcommand{\J}{{\Bbb J}}
\newcommand{\Ll}{{\Bbb L}}
\renewcommand{\L}{{\Bbb L}}
\newcommand{\M}{{\Bbb M}}
\newcommand{\N}{{\Bbb N}}
\newcommand{\Q}{{\Bbb Q}}
\newcommand{\R}{{\Bbb R}}
\newcommand{\T}{{\Bbb T}}
\newcommand{\V}{{\Bbb V}}
\newcommand{\U}{{\Bbb U}}
\newcommand{\W}{{\Bbb W}}
\newcommand{\X}{{\Bbb X}}
\newcommand{\Z}{{\Bbb Z}}

\newcommand{\tT}{{\mathrm{T}}}
\newcommand{\tD}{{D}}
\newcommand{\hyp}{{\mathrm{hyp}}}
\newcommand{\cusp}{{\mathrm{cusp}}}

\newcommand{\fix}{{b}}
\newcommand{\cxfix}{{\xi}}
\newcommand{\LINV}{L_{\rm inv}}
\newcommand{\LLINV}{{\mathcal L}_{\rm inv}}
\newcommand{\f}{{\bf f}}
\newcommand{\g}{{\bf g}}
\newcommand{\h}{{\bf h}}
\renewcommand{\i}{{\bar i}}
\renewcommand{\j}{{\bar j}}

\newcommand{\geoD}{{\boldsymbol{\gamma}}}

\newcommand{\Bf}{{\mathbf{f}}}
\newcommand{\Bg}{{\mathbf{g}}}
\newcommand{\Bh}{{\mathbf{h}}}
\newcommand{\Bi}{{\mathbf{i}}}
\def\BJ{{\mathbf{J}}}
\def\Bl{{\mathbf{l}}}
\def\Bm{{\mathbf{m}}}
\def\Bn{{\mathbf{n}}}

\def\Bj{{\mathbf{j}}}
\def\BJ{{\mathbf{J}}}
\def\Bphi{{\mathbf{\Phi}}}
\def\Bpsi{{\mathbf{\Psi}}}
\newcommand\Bom{\boldsymbol{\om}}

\newcommand{\BD}{{\boldsymbol{D}}}
\newcommand{\BE}{{\boldsymbol{E}}}
\newcommand{\BF}{{\boldsymbol{F}}}
\newcommand{\BG}{{\mathbf{G}}}
\def\BH{{\mathbf{H}}}
\newcommand{\BI}{{\boldsymbol{I}}}
\newcommand{\BK}{\mathbf{K}}
\newcommand{\BL}{\mathbf{L}}
\def\B0{{\mathbf{0}}}
\newcommand{\BP}{{\boldsymbol{P}}}
\newcommand{\BS}{{\boldsymbol{S}}}
\def\BT{{\mathbf{T}}}
\newcommand{\BW}{{\mathbf{W}}}
\newcommand{\BV}{{\mathbf{V}}}
\newcommand{\BOM}{{\boldsymbol{\Om}}}

\newcommand{\BPi}{{\boldsymbol{\Pi}}}
\def\BUps{{\boldsymbol{\Upsilon}}}
\def\BLa{{\boldsymbol{\La}}}
\def\BGa{{\boldsymbol\Gamma}}
\def\BDe{{\boldsymbol\Delta}}
\def\BUps{{\boldsymbol\Upsilon}}
\def\BThe{{\boldsymbol\Theta}}
\def\BOm{{\boldsymbol \Om}}
\def\BPsi{{\boldsymbol\Psi}}

\def\Baleph{{\boldsymbol\aleph}}

\newcommand{\Bw}{{\mathbf{w}}}
\newcommand{\Bal}{{\boldsymbol{\alpha}}}
\newcommand{\Bde}{{\boldsymbol{\delta}}}
\newcommand{\Bga}{{\boldsymbol{\gamma}}}
\newcommand{\Bsi}{{\boldsymbol{\sigma}}}
\newcommand{\Bla}{{\boldsymbol{\lambda}}}
\def\Be{\mathbf{e}}
\def\Dia{{\Diamond}}

\def\SB{{\boldsymbol{\BB}}}

\def\ext{{\mathrm{ex}}}
\def\mouth{\operatorname{mouth}}
\def\tail{\operatorname{tail}}

\newcommand{\Comb}{{\it Comb}}
\newcommand{\Top}{{\operatorname{Top}}}
\newcommand{\Bottom}{{\operatorname{Bot}}}
\newcommand{\QC}{\mathcal QC}
\newcommand{\Def}{\mathcal Def}
\newcommand{\Teich}{\mathcal Teich}
\newcommand{\PPL}{{\mathcal P}{\mathcal L}}
\newcommand{\Jac}{\operatorname{Jac}}
\newcommand{\Homeo}{\operatorname{Homeo}}
\newcommand{\AC}{\operatorname{AC}}
\newcommand{\Dom}{\operatorname{Dom}}
\newcommand{\ord}{\operatorname{ord}}

\newcommand{\Hol}{{\rm Hol}}

\newcommand{\Aff}{\operatorname{Aff}}
\newcommand{\Euc}{\operatorname{Euc}}
\newcommand{\MobC}{\operatorname{M\ddot{o}b}({\Bbb C}) }
\newcommand{\PSL}{ {\mathcal{PSL}} }
\newcommand{\SL}{ {\mathcal{SL}} }
\newcommand{\CP}{ {\Bbb{CP}}   }

\newcommand{\hf}{{\hat f}}
\newcommand{\hz}{{\hat z}}
\newcommand{\hM}{{\hat M}}

\renewcommand{\lq}{``}
\renewcommand{\rq}{''}

\newcommand{\Ch}{\textrm{Ch}}


\catcode`\@=12

\def\Empty{}
\newcommand\oplabel[1]{
  \def\OpArg{#1} \ifx \OpArg\Empty {} \else
  	\label{#1}
  \fi}
		
%

\long\def\realfig#1#2#3#4{
\begin{figure}[htbp]
\centerline{\psfig{figure=#2,width=#4}}
\caption[#1]{#3}
\oplabel{#1}
\end{figure}}

%

\newcommand{\comm}[1]{}
\newcommand{\comment}[1]{}

\title[Elephant Eyes]{A priori bounds for some\\ infinitely renormalizable quadratics:\\
  {\tiny IV. Elephant Eyes } }

\author {Jeremy Kahn and Misha Lyubich}

\bigskip\bigskip

\date{}

\begin{abstract} 
  In this paper we prove {\it a priori} bounds for  an ``elephant eye'' combinatorics.
  Little $M$-copies specifying these
  combinatorics are allowed to converge to the cusp of the Mandelbrot set.
  To handle it, we develope a new geometric tool:
  uniform thin-thick decompositions for bordered Riemann surfaces.   
\end{abstract}

\setcounter{tocdepth}{1}
 
\maketitle
\tableofcontents

\section {Preamble}

The main result of this paper was obtained by the authors around 2010
and has been announced at a number of conferences,
beginning with the conference
``Geometric and Algebraic Structures in Mathematics''
(Stony Brook 2011):

\msk
https://www.math.stonybrook.edu/Videos/dfest/video.php?f=48-Kahn
\msk

Here we reproduce the version of the article from  Jan 2020.

\section{Introduction}

The MLC Conjecture asserting that the Mandelbrot set $M$  is locally connected is a
central long-standing problem in contemporary Holomorphic Dynamics.
A major breakthrough in the problem occured around 1990
when Yoccoz proved local connectivity of $M$ at any parameter which is not infinitely renormalizable.

Since then, many infinitely renormalizable parameters has been taken care of,
see \cite{puzzle,K,decorations,molecules}.
However, for all of them so far, the renormalization types allowed
belong to finitely many limbs of $M$. In this paper, we cover an unbounded situation,
where the combinatorial types belongs to arbitrary  $1/\qq$-limbs of $\MM$
with renormalization periods $p= \qq+O(1)$.

To this end, we develope a new geometric tool:
uniform thin-thick decompositions for bordered Riemann surfaces,
which has a potential to  find many other applications.   

\section{Canonical laminations}

\subsection{Ideally marked hyperbolic plane}

\sss{Quadrilaterals}
  Let $\BPi$ be a  standard rectangle with horizontal sides $\BI$ and
  $\BJ$. It supports  the vertical foliation $\FF$. We let 
$$
   \WW(\BI, \BJ) =  \WW(\BPi)  = \WW(\FF) 
$$
be its conformal {\it width}, which is also equal to the width of the 
{\it full path family} in $\Pi$ connecting $\BI$ to $\BJ$.  

\begin{rem}
  We will sometimes use notation $\WW\{ \BI, \BJ\}$ to emphasize that the
  pair of intervales $\BI$ and $\BJ$ is non-ordered. 
\end{rem}

Let $\Bga$ and $\Bde$ be the hyperbolic geodesics in $\Pi$ sharing
the endpoints with $\BI$ and $\BJ$ respectively. We let $d_\hyp\equiv d_\Pi$ be the
hyperbolic distance in $\BPi$, and $l_\hyp\equiv l_\BPi$ be the
 hyperbolic length  (similar notation will be used for other
hyperbolic Riemann surfaces). 

\begin{lem}\label{geod in quadr}
    Let $\eps>0$ and let $\Bga^\eps=\{ z\in\Bga: \ \dist_\hyp (z, \Bde)
    < \eps\} $. Then 
$$
       l_\hyp (\Bga^\eps) = \WW(\BPi) + O_\eps (1).
$$
\end{lem}

Let us now formulate an important transformation rule: 

\begin{lem}\label{tranform rule for rectangles}
   Let $\Bi: \BPi'\ra \BPi$ be a holomorpic map between two rectangles
which extends to homeomorphisms between the respective  horizontal  sides
$\BI'\ra \BI$ and $\BJ'\ra \BJ$. Then
$$
      \WW(  \BPi' ) \leq \WW ( \BPi ). 
$$ 
\end{lem}

\begin{proof}
Let $\De$ and $\De'$ be the horizontal path families in $\BPi$
and $\BPi'$ respectively.  
 Any  path  $\de$ in $\BPi$  contains a subpath that lifts to a 
  path in $\De'$.  Hence 
$$
     \WW(\BPi)= \LL(\De) \geq \LL(\De') = \WW(\BPi').
$$
\end{proof}

\sss{Canonical lamination}
 
  Let us consider the hyperbolic plane in the disk model $\D$. 
An {\it ideal marking } of $\D$ is a {\em tiling}  $\II$
of the circle  $\T=\di \D$ into intervals $\BI_1, \dots \BI_p$
(written in the order they appear on $\T$, up to a cyclic permutation). 
Thus, we can label the intervals by elements of $\Z/ p\Z$. 
We call $p$ the {\it cardinality } of the marking. 

Any non-adjacent pair of intervals $\BI_m$, $\BI_n$
determines an ideal quadrilateral $\ol \BPi_{mn}$ with ``horizontal'' sides $\BI_m $ and  $\BI_n $.
Let $\ol \FF_{mn}$ be the vertical harmonic foliation in $\ol \BPi_{mn}$,
and let  
$$
   \ol \WW_{mn}\equiv \ol \WW(\BI_m, \BI_n) \equiv \WW (\ol  \BPi_{mn} ) \equiv \WW(\ol \FF_{mn} )
$$ 
be the width of this quadrilateral.
It is equal to the width of the full path family in $\D$  connecting $\BI_n$ to $\BI_m$. 

Removing  the square buffers from the  quadrilaterals $\ol \BPi_{mn}$
with width  greater than 2,
we obtain quadralaterals $\BPi_{mn}$ supporting foliations $\FF_{mn}$
with widths
$$
  \WW_{mn} \equiv  \WW(\BI_m, \BI_n) \equiv \WW (\BPi_{mn}  ) \equiv
  \WW(\FF_{mn} )  = \ol \WW_{mn} -2. 
$$

\begin{rem}
   Again, we sometimes wlll use notation  $\WW_{\{ mn\} } \equiv
   \WW\{\BI_m, \BI_n\}$ to emphasize that the pair $\{m,n\}$ is not
   ordered. 
\end{rem}

\begin{lem}\label{canon fol}
    The foliations $\FF_{mn}$ together  form a lamination
    $\FF(\D,\II)$   on $\D$ (supported on $\bigsqcup \BPi_{mn}$).  
\end{lem}

\begin{proof}
  Follows from the Non-Intersection Principle.
\end{proof}

The lamination $\FF(\D,\II)$ is called the {\it canonical lamination
of the  marked hyperbolic plane}. 

\begin{cor}
   There exits at most $p-2$ quadrilaterals $\BPi_{\{mn\}} $.
\end{cor}

\sss{Conformal geometry vs hyperbolic}

Let $\Bga_n$ be the hyperbolic geodesic in $(\D, \II)$ sharing the
endpoints with the $\BI_n$.
Let
$$
\Bga^\eps_{mn}= \{ z\in\Bga_m : \ \dist_\hyp (z, \Bga_n) < \eps\}.
$$
%
%

Lemma \ref{geod in quadr}  implies: 

\begin{lem}\label{hyp d vs length}
   We have: $\WW_{mn} = l_\hyp(\Bga_{mn}^\eps )   + O_\eps (1). $
\end{lem}

\sss{Thick part}

let $\BGa= \cup \Bga_n$. 
It is  the polygonal curve that bounds  the convex hull $\BP$  of $\di \II= \bigcup \di \BI_n$.  
Let $ \BGa^\eps = \bigcup \Bga^\eps_{mn}$ be the $\eps$-{\it thin part} of $\BGa$. 

\begin{prop}\label{thick part estimate}
 For an ideally  marked hyperbolic plane $(\D, \II)$ with $p$
 intervals $I_n$ and for any $\eps>0$, we have:
$$
       l_\hyp (\BGa \sm \BGa^\eps) = O(p/\eps).
$$
\end{prop}

\begin{proof}
  Let $L= l_\hyp (\BGa \sm \BGa^\eps) $.
  Let us consider a biggest net of $\eps$-separated points $\{ x_i\}$ in $\BGa\sm
  \BGa^\eps$. It contains approximately  $ L/\eps$ points. The
  hyperbolic  disks $\BD_i$  of
  radius $\eps/2$ centered at these points do not overlap, so the total
  hyperbolic area of the half-disks $\BD_i \cap \BP$  is bouded by the hyperbolic area of the polygon
  $\BP $, which is equal to  $\pi (p-2)$. It gives
$$
        \frac L\eps \cdot \eps^2  \leq C  p.
$$
\end{proof}

 \sss{Thin-thick Decomposition}

Let $\ol  \WW_n \equiv \bar \WW(\BI_n)$ be the width of the path family connecting $\BI_n$ 
all other non-adjacent intervals (i.e., to  $\displaystyle{ \T\sm \bigcup_{|k-n| >1} \BI_k } $),
and let 
$$
           \WW_n\equiv    \WW(\BI_n) =   \max  \{ \ol \WW(\BI_n)  -2, \ 0 \} .
$$ 
We call
$$
      \BW \equiv \BW(\D, \II ) :  = \sum \WW_n 
$$
the {\it total weight} of the marked hyperbolic plane. 

\begin{thm}\label{main estimate} 
For any ideal marking $\II = (\BI_1, \dots, \BI_p)$ of the hyperbolic plane
$\D$,  we have:
$$
         \BW (\D, \II)   - O(p)    \leq   2 \sum \WW\{\BI_n, \BI_m\}  
     \leq    \BW(\D, \II ) 
$$
$($where the summation in the middle is taken over non-ordered pairs
$\{ \BI_m, \BI_n\}$$)$. 
\end{thm} 

\begin{proof}
  Since the rectangles $\BPi_{mn}$ are disjoint, we have for each $m$: 
$$
    \sum_n \WW_{\{mn\}} \leq \WW_m.
$$
Summation over $m$ gives the right-hand side estimate (note that each
$\WW_{ \{mn\} }$ is counted twice as we sum  them  up). 

\msk To prove the left-hand estimate, let us take $\eps$ to be the
Margulis constant. Then the  geodesic intervals $\Bga^\eps_{mn}$ are
pairwise disjoint.  Let 
$$
   \BGa^\thin = \bigsqcup_{|m-n|>1} \Bga_{mn}^\eps 
$$
(where $|m-n|$ means the natural distance in $\Z/ p\Z$ such that $m$
and $m+1$ stay distance 1 apart).  
Together with Lemma \ref{hyp d vs length} this implies:

$$
     2\sum \WW_{ \{mn\} } \geq \sum_{|m-n|>1} l_\hyp (\Bga_{mn}^\eps ) - O(p)
     =  l_\hyp ( \BGa^\thin  ) - O(p).
$$

Combining this with Lemma \ref{thick part estimate}, we obtain: 
\begin{equation}\label{thin-thick part}
     2\sum \WW_{ \{mn\} } \geq   l_\hyp ( \BGa^\thin \cup \BGa^\thick ) - O(p).
\end{equation}

Let us now consider the cuspidal  part of $\BGa$:  
$$
 \BGa^\cusp=  \bigcup_{|m-n|=1} \Bga^\eps_{mn}.
$$
We have a partition: 
$$
   \BGa =    \BGa^\thin \cup \BGa^\cusp \cup \BGa^\thick. 
$$
Incorporating it into (\ref{thin-thick part}), we obtain:
$$
    2\sum \WW_{ \{mn\} } \geq   l_\hyp ( \BGa\sm \BGa^\cusp ) - O(p).
$$

Finally, let $\Bga_m^\eps$ consists of those points $z\in \Bga_m$ that
stay distance less than $\eps$ from the geodesic connecting the
endpoints of $\Bga_{m-1}\cup \Bga_m \cup \Bga_{m+1}$. Since $\eps$ is
the Margulis constant, $\Bga_m^\eps \subset \Bga_m\sm \BGa^\cusp$, 
so the last estimate implies
$$
       2\sum \WW_{ \{mn\} } \geq  \sum  l_\hyp ( \Bga_m^\eps ) - O(p).
$$
By Lemma \ref{hyp d vs length},
$ l_\hyp ( \Bga_m^\eps ) = \WW_m + O(1)$, and we are done.
\end{proof}

\sss{Local weights and fluxes for a partially marked plane}

Let $\II= (\BI_1, \dots, \BI_p)$ be a disjoint family of closed
intervales in $\T$.
 It is called a {\it partial marking} of $\D$ (of cardinality $p$). 
The {\it gaps} $G_k$  of a partially marked plane are the closures of the
components of $\displaystyle{ \T \sm \bigcup \BI_n}$. Of course, putting
together the intervals and the gaps, we obtain a full ideal  marking
of $\D$. The {\it local weights} $\WW(\BI_n)$ and {\it fluxes}
$\WW(\BG_k)$ are understood in the sense of this full marking.     
 
Let us consider two partially marked copies of the hyperbolic planes, $(\D, \II)$ 
and $ (\D', \II')$, of the same cardinality. We write 
$\Bi : (\D',\II')  \ra (\D, \II) $ for a map $\Bi: \D' \ra \D$
that  extends to a  homeomorphism   $\BI_n' \ra \BI_n$ for any  $n=1,\dots, p$.   
In this case, we have an obvious consistent labeling  of the gaps
$\BG_k$ and $\BG_k'$. 

\begin{lem}\label{transform rules on D}
  For any holomorphic map $\Bi: (\D', \II') \ra (\D, \II) $ between
  partially marked hyperbolic disks we have:

\ssk\nin $\mathrm { (i) } $
$
   \WW(\BI_m' , \BI_n')  \leq \WW(\BI_m, \BI_n). 
$

\ssk\nin $ \mathrm {(ii)} $  $\WW(\BG_k' ) \geq \WW(\BG_k)$. 
\end{lem}

\begin{proof}
  The first estimate follows directly from 
Lemma \ref{tranform rule for rectangles} .  For the second one, observe that
that $\WW (\BG_k)= \WW (\BI_m, \BI_n )^{-1}$, where $\BI_m$ and
$\BI_n$ are the marked intervals adjacent to the gap $\BG_k$.  
\end{proof}

\subsection{Riemann surfaces with boundary}

\sss{Universal covering}

Let us now consider a compact Riemann surface $S$ with non-empty
boundary $\di S = \bigsqcup J_k$. We let $\pi: \D \ra \inter S$ be the 
universal covering, 
and we let $\La\subset \T$ be the limit set for the  group $\De$
of deck transformations.  Then $\pi$ extends continuously to $\T\sm
\La$, and
$$
    \pi:  \ol \D \sm \La \ra  S 
$$ 
is the  universal covering of $S$. Moreover, $\pi$ restricted to  any component of
$\T\sm \La$  gives us a universal covering of some component $J_k$ of
$\di S$.  Such a component of $\T\sm \La$ will be denoted $\BJ_k$
(usually we will need only one component for each $k$, 
so we will not use an extra label for it). 
The stabilizer of $\BJ_k$ in $\De$ is a cyclic groups generated by
a hyperbolic M\"obius transformation $T_k$ with fixed points in  $\di \BJ_k$.

\sss {Local weghts}

We let $\A_k$ be the covering annulus of $S$ corresponding to the
cyclic group $<T_k>$. We call its width $\WW(\A_k)$ 
the {\it local  weight} of $J_k$. We will also use notation $\WW(J_k)$
for this weight.  

 We let
$$
  \BW(S) = \sum \WW(J_k)
$$
be the {\it total weight} of $S$. 


\sss{Canonical arc diagram}
  Here we will summarize some notions and results  from \cite{K}. 

  An {\it arc}  $\alpha$ is a proper path in $S$ up to proper homotopy. 
Any arc connects some components $J_k$ and $J_i$ of $\di S$
and lifts to an arc $\ba$ in $\D$ connecting some intervals $\BJ_k$
and $\BJ_i$. We let  $\ol\WW(\alpha)= \ol\WW(\BJ_k, \BJ_i)$.
It is the width of the quadrilateral $\ol \BPi_\alpha=\ol \D$ with
horizontal sides $\BJ_k$ and $\BJ_i$.  
It is independent of the lift used. 

If $\ol\WW(\alpha)> 2$ then
by removing from $\ol\BPi_\alpha$ the square buffers we obtain a
rectangle $\BPi_\ba$.  In this case we let
$$
      \WW(\alpha) = \WW(\BPi_\ba)  = \ol \WW(\alpha)-2.
$$ 
If $\bar\WW(\alpha)\leq 2$ we let  $\WW(\alpha) =0$.

Non-degenerate quadrilaterals $\BPi_\alpha$ project injectively to 
disjoint  {\it canonical quadrilaterals} $\Pi_\alpha\subset S$.  
Each of them supports the canonical
vertical foliation. Altogether these foliations form the
{\it canonical lamination} of $S$. 

Arcs with positive weight form the {\it canonical arc diagram}
$\AAA_S$ on $S$.


\sss{Thin-thick decomposition}

\begin{thm}\label{thin-thick for S}
   Let $S$ be  a compact Riemann surface with boundary,
$\chi=\chi(S)$ be its Euler characteristic. Then  we have:
$$
    \BW(S)  -  O( | \chi | ) \leq   2 \sum_{\alpha\in \AAA_S}  \WW(\alpha)  \leq \BW(S). 
$$
\end{thm}

\begin{proof}
    For any component $J_k$ of $\di  S $, $\WW(J_k)\geq \sum
    \WW(\alpha)$, where the summation is taken over the arcs that land
    on $J_k$. This implies that right-hand side estimate. 

 To prove the other one, let us consider the simple closed geodesics
 $\gamma_k$ homotopic to the $J_k$. Then 
\begin{equation}\label{geod and annuli} 
        \WW_k =  \frac 1\pi\,   l_\hyp(\gamma_k) \quad \mathrm{so}\ 
\BW  (S) = \frac 1\pi \sum l_\hyp(\gamma_k). 
\end{equation}

Let $P$ be the Riemann surface bounded the geodesics $\gamma_k$. 
As in the proof of Proposition \ref{thick part estimate}, let $\eps$ be the Margulis
constant for $P$ (equal to the Margulis constant for the double of $P$).

Take some arc $\alpha$ connecting $J_k$ to $J_i$, and lift it to the
universal covering. We obtain an arc $\ba$ in $\D$ connecting $\BJ_k$
to $\BJ_i$. Moreover, the geodesics $\gamma_k$ and $\gamma_i$ lift to
the geodesics $\Bga_k$ and $\Bga_i$ in $\D$ sharing the endpoints with $\BJ_k$ and $\BJ_i$
respectively. By Lemma \ref{geod in quadr}, 
\begin{equation}\label{rectangle  and geod}   
   \WW(\alpha) \equiv \WW(\ba) =\frac 1\pi\, l_\hyp (\Bga_{ki} ^\eps)
   + O(1)= \frac 1\pi \, l_\hyp (\gamma_{ki} ^\eps) +O(1),
\end{equation}
where $\gamma_{ki}^\eps$ are defined on $S$ in the same way as the
$\Bga_{ki}^\eps$ are defined in $\D$, and in fact, the latter is the
lift of the former.  

Let 
$$
   \gamma_k^\thick: =  \gamma_k \sm \bigsqcup_i \gamma_{ki}^\eps. 
$$
Then the Gauss-Bonnet  argument (as in Proposition \ref{thick part estimate}) shows that 
$$
    \sum l_\hyp ( \gamma_k^\thick ) = O(\area (P) ) = O(|\chi | ).
$$
Putting this together with (\ref{geod and annuli}) and (\ref{rectangle  and geod}), we obtain the
desired. 
\end{proof}


\subsection{Marked Riemann surfaces with boundary}
\note{need?}

Two particular cases that we have considered above can be unified in
one frame. 

\sss{Boundary marking}
Assume now that some boundary components $J^k$ of $S$ are   
partitioned into several intervals, 
$\II^k= (\BI_1^k, \dots, \BI_{p_k}^k)$.
If some component $J^k$ is not
partitioned, we let $J^k\equiv I^k_1$  and  $p_k=1$.
We say that $S$ is endowed with  the boundary marking $\II=\bigsqcup I^k$.
We let $|\II| = \sum p_k$ be the {\it cardinality} of the marking.  
 
An {\it arc} $\alpha $ on a marked surface is a path connecting some
{\it non-adjacent }  marked 1-manifolds $I^k_m$ and $I^j_n$, 
up to proper homotopy.  Lifting $\alpha $ to the universal covering $\D$,
we obtain an arc $\Bsi$ connecting two ideal intervals $\BJ^k_n$ and
$\BJ^i_m$.
Then we define the weight of $\alpha$ as

$$
   \WW(\alpha)= \WW (\BJ^k_n, \BJ^j_m). 
$$

Arcs with positive weights form a weighted arc diagram $\AAA(S, \II)$
of the marked surface.

Let us also define {\it local weights} $\WW( J^k_n)$.
If $J^k_n$ is a circle then the definition is the same is in the
unmarked case. Otherwise we define it as the width of the recatnge in
the universal covering with horizontal sides $\BJ^k_n$ and the
complement in $\T$ of $\BJ^k_n$ together with two adjacent lifted intervals. 

The {\it total weight} of a marked surface is defined as
$$
   \BW(S, \II) = \sum_\II \WW( J^k_n ).
$$

\sss{Thin-thick decomposition}

\begin{thm}\label{thin-thick for marked surfaces}
    Let $(S, \II)$ be a marked  compact Riemann surface with bundary,
$\chi=\chi(S)$ be its Euler characteristic. Then
$$
  \BW(S, \II) - O( | \chi |  + |\II|) \leq 2\, \sum_{\AAA(S, \II) } \WW(\alpha ) \leq
  \BW(S, \II).
$$   
\end{thm}

\note{Use total weight notation}

\begin{proof}
  Combination of the above two arguments. 
\end{proof}

\sss{Proper domains}
Let $S$ be  a compact Riemann surface $S$ with  boundary, and 
let $D$ be a   
domain in $S$ (defined up to homotopy) bounded (rel $S$) by several
ideal intervals  $\tl J_n\subset  J_n $ and several arcs $G_j$ (called {\it edges}).
 We  call $D$ a {\it proper domain}.)
 

Let $\BD$ be 
 the covering of $S$ corresponding to  $D$. It is  a partially marked
 Riemann surface $(\BD, \II)$  with  marked intervals  $\BJ^k_n$
 covering the $J_n$,
 together with other components of $\di D$.

\bignote{Pay attention to  the outer boundary}
 
This determines local weights $\WW(\tl J_n) \equiv \WW(J_n) : = \WW (\BJ_n)$,
{\it  fluxes} $\WW^\perp (G_j):= \WW(\BG_j)$, and the total weight of $D$:
$$
   \BW(D)  : = \sum \WW (J_n) + \sum \WW^\perp (G_j).  
$$


\msk
We can also consider arcs $\alpha$ in $D$
connecting various pairs of intervals $\tl J_n$ and $\tl J_m$.
Letting   $\WW(\alpha):= \WW(\ba)$, 
where $\ba$ is the  lift of $\alpha$ to $\BD$, we obtain the weighted arc
diagram $\AAA_D$.

Along with arcs, we can consider {\it segments} $\si$ connecting two
edges $G_j $ and $G_i$.  Each segment is lifted to an arc $\Bsi$ in $\BD$
connecting the corresponding  gaps  $\BG_j$ and  $\BG_i$. 
Letting $\WW(\si) := \WW(\Bsi)$, we obtain  the  {\it weighted segment diagram}
$\SSS_D$. 

We  can also consider {\it arc-segments} in $D$ connecting intervals
$\tl J_n$ to  gaps $G_j$, and to define their weights.
They form the  {\it weighted  arc-segment diagram} $\AS_D$. 
Altogether,
arcs, segments and arc-segments form the  {\it weighted  arc-segment diagram}
$\SAS_D$. 

\ssk
Summing  up the  weights of arcs, segments, or arc-segments,
we obtain the {\em total weights} of the corresponding diagrams,
$\BW(\AAA_D)$, $\BW (\SSS_D)$, and $\BW(\AS_D)$.
{We will use the same notation, $\BW(\QQ)$, for any famly of
 arc-segments.)  
Theorem \ref{main estimate}
implies:

 \begin{cor}\label{thin-thick for domains}
  Let  $D\subset S$ be  a proper $p$-gon as 
   Then  
 $$
      \BW(D) - O( p)  \leq   2\, \sum_{\AS_D}  \WW(\si)    \leq  \BW(D). 
 $$
 \end{cor}

\comm{*****
In fact, for any arc $G$ in $S$  we can define the flux 
$\WW^\perp(G)$ through $G$ as follows. 
Let $G$ land on boundary components $J_n$ and $J_m$.
Then it can be lifted to an arc $\BG$ in $\D$ landing on $\BJ_n$ and
$\BJ_m$.
This determines an ideal  quadrilateral with {\it vertical} sides $\BJ_n$
and $\BJ_m$. Its width is $\WW^\perp(G)$.  

Similarly, given two arcs $G_i$ and $G_j$ 
(that could be equal: $G_i=G_j$)
 connected by  a segment
$\si$ in $S$, we can define the weight $\WW(\si)$  as follows. 
Let $G_i$ land on $J_n$ and $J_m$, while $G_j$ land
on $J_s$ and $J_t$. Then we can lift the whole configuration to $\bar\D$
marking four  $\BJ$-intervals on $\T$. The segment $\Bsi$ connects two
gaps in this configuration that should be viewed as the horizontal
sides of a quadrilaterlal. Then $\WW(\si)$ is the width of this quadrilateral.

Of course, we can deine the weight of a general arc-segment in $S$ as well. 

******* } 

\subsection {Transformation rules}

By passing to the Universal Covering (or to the annulus covering), 
transformation rules of  Lemma \ref{transform rules on D} readily
generalize from the disk to arbitrary marked Riemann surfaces: 

\begin{lem}\label{generaltransform rules}
Let $i: S' \ra S$ be a  holomorphic map between
two compact  Riemann surfaces with boundary.
Then 

\ssk\nin $ \mathrm{(i)} $  If a boundary component $J' $ of $S'$ is mapped
with degree $d$ to a boundary component $J $ of $S$ then 
$ \WW(J')\leq d\cdot \WW(J)$;

\ssk\nin $ \mathrm{(ii)} $  If a  gap  
$G ' $ of $S'$ is mapped
to a   gap  
$G $ of $S$ then $\WW^\perp (G')\geq \WW^\perp (G)- 2 $;

\bignote{Used to be a ``segment'' instead of a ``gap''.
  Also: Notation``$\si$'' is used for segments} 

\ssk\nin $\mathrm { (iii) } $
 If an arc $a ' $ of $S'$ is mapped
to an arc  $\alpha $ of $S$ then $\WW (\alpha')\leq \WW (\alpha)$;

\ssk 
In case of a covering map, all the above inequalities become
equalities (without  subtracting ``$2$'' in $\mathrm { (ii) } $). 
\end{lem}


\bignote{Extension to arc diagrams is hidden here} 

\comm{*****
We can also naturally define the flux $\WW^\perp (X)$  through a weighted arc diagram 
$X= \sum x_i G_i$, namely
$$
         \WW^\perp (X) =\sum x_i \WW^\perp (G_i).
$$
By linearity, property (ii) of Lemma \ref{generaltransform rules}
extends to this more general  situation:

\begin{cor}\label{flux through arc diag}
  Under the circumstances of the above lemma,
if  a weighted  arc diagram $X ' $ of $S'$ is mapped by $i$
to a weighted  arc diagram   $X $ of $S$ then 
$$
  \WW^\perp (X')\geq \WW^\perp (X)-O(|X'|),
$$
with the equality (with no error term) in case of the covering map.
\end{cor}
*****************}

Finally, let us also consider the case of concatenation of several
arcs. By Lemma \ref{generaltransform rules} \note{specify} 
and  the Parallel Law,we have:

\begin{lem}\label{transform rule for concatenation}
Let $i: S' \ra S$ be a  holomorphic map between
two compact  Riemann surfaces with boundary.
Let $G_1', \dots , G_n'$ be a family of arcs on $S'$ 
whose concatenation through some imporper boundary components $J_k'$ 
forms an arc $G$ in $S$. Then
$$
   \sum \WW^\perp (G_k') + \sum \WW( J_k') \geq \WW^\perp(G) - O(n). 
$$
\end{lem}

\sss{Key estimate} 

Let $S$ and $S'$ be two Riemann surfaces, and let $D\subset S$ and
$D'\subset S'$ be two proper domains, endowed with their
segment diagrams, $\SSS_D$ and $\SSS_{D'}$.

Let $i: S' \ra S $ be a holomorphic map such that $i(D')$ is properly
  isotopic to $D$.  Given any  segment $\sigma \in \SSS_D$, the pullback
  $i^*(\si)$ is a   concatenation of several  segments $\si_k'\in \SSS_{D'}$.  \note{picture?} 
  We say
  that $\si$ {\it breaks} under the pullback if $i^*(\si)$ comprises more
  than one segment. Let $\SSS_D^\br$ stand for the set of segments that break.     

  

\begin{lem}\label{key estimate} 
  Under the above circumstances,  we have:
$$
     \BW (\SSS_D^\br) \leq \BW (\SSS_{D'}) - \BW (\SSS_D)  + O(p+|\chi|). 
$$
\end{lem}

\begin{proof}
  By Maximality of $\SSS_{D'} $ we have:
$$
\BW(\SSS_{D'}) \geq \BW  ( i^*(\SSS_D) ) - O(p+|\chi|)
  \geq \BW (\SSS_D)  + \BW(\SSS_D^\br) - O(p+|\chi|).
$$
The last estimate follows from the fact that    $\WW(\si') \geq \WW(\si)$
for any lift $ \si_k'\in \SSS_{D'}  $ of  a segment $\si\in \SSS_D$. Since
broken segments  admit at least two lifts, their weights are counted
at least twice in $i^* (\SSS_D) $, so
$ \WW(i^*(\SSS_D)) \geq \WW(\SSS_D)+\WW(\SSS_D^\br)$.  
\end{proof}

\comm{*****
\begin{lem}\label{key estimate} 
  Under the above circumstances,  we have:
$$
     \WW(\FF^\br) \leq \BW(D', \II') - \BW (D, \II)  + O(p+|\chi|). 
$$
\end{lem}

\begin{proof}
  By Maximality of $\FF'$, we have:
$$
   \BW(D', \II') \geq \WW( i^*(\FF) ) - O(p+|\chi|) \geq \BW (D, \II)
   + \WW(\FF^\br)
   - O(p+|\chi|).
$$
The last estimate follows from the fact that    $\WW(\si') \geq \WW(\si)$
for any lift $ \si' $ of  a segment $\si$ in $D$. Since
broken segments $\si$ admit at least two lifts, their weights counted
at least twice in $i^* (\FF)$, so
$ \WW(i^*(\FF)) \geq \WW(\FF)+\WW(\FF^\br)$.  
\end{proof}
************************} 

We can similarly define arcs and arc-segments in $D$ that break,
$\AAA_D^\br$ and $\AS_D^\br$.
Let us say that an arc in $D$  is {\em short} if it connects 
 two consecutive intervals, and {\em long} otherwise.
 Let $\AAA_D^\sh$ and $\AAA_D^\lo$ stand respectively for the
 familes of those arcs, and let $\AAA_D^\lobr$ stand for the
 family of long arcs that break.

 \begin{lem}\label{broken arc-segments} 
   Under the above circumstances, we have:

$$
\BW(\AS_D^\br) \leq \BW (\SSS_{D'} ) \quad {\mathrm {and}} \quad
\BW(\AAA_D^\lobr) \leq \BW (\SSS_{D'} ) + \BW (\AS_{D'} )  .
$$
   
 \end{lem}

 \begin{proof}
The pullback $i^*(\alpha) $ of any arc-segment  $\alpha\in \AS_D^\br$
 consists of at least two pieces, one of which must be a
 segment in $D'$. Moreover, the weight of this segment is at least
 $\WW(\alpha)$. Hence
 $$
 \BW(\SSS_{D'} ) 
 \geq \sum_{\alpha\in \AS_D^\br} \WW (\alpha) = \BW(\AS^\br).  
 $$

 \ssk 
Similarly,    the pullback $i^*(\alpha)$  of any  long arc
$\alpha\in \AAA^\lobr $ that breaks consists of

\ssk\nin $\bullet$ either two pieces, one which must be an arc-segment
in $D'$;

\ssk\nin $\bullet$
or at least three pieces, one of which (a ``middle'' one)  must be a
segment in $D'$.

\ssk The second desird estimate follows. 
\end{proof}

\comm{****** 

\subsection{Closed curve diagram}

\sss{Definition}
Given a non-trivial and non-cuspidal  simple closed curve $\gamma$ on a Riemann surface $S$,
let  $\A_\gamma$ be  the covering annulus  corresponding to $\gamma$,,
and let  $\Bga$ be the closed geodesic homotopic to $\gamma$.
Then we let 
$$
    \WW(\gamma)= \mod \A_\gamma= \pi/  l_\hyp(\Bga).
$$

For $M$ sufficiently big (``4''), 
the $M$-{\it canonical curve diagram} $\CC$ is the family of simple
closed curves $\gamma$ with $\WW(\gamma) > M$. 

\begin{lem}\label{disjointness in CC}
  The curves is $\CC$ are pairwise disjoint.  
\end{lem} 

\sss{Transformation rules} 

\begin{lem}\label{tranform rules for CC}
  Let $i: S' \ra S$ be a holomorphic map between Riemann surfaces.
Let $\gamma\in \CC$ and $\gamma'\in \CC'$  be simple closed curves
such that $i(\gamma')$ is homologous to  $d\cdot[\gamma$. 
 Then $\WW(\gamma') \leq d^{-1} \cdot \WW(\gamma)$.
In the case when $i$ is a covering, the equality holds. 
\end{lem}

\begin{proof}
We can assume that $\gamma'$ is a closed geodesic.
By the Schwarz Lemma, 
$$
     l_\hyp(\gamma') \geq  l_\hyp (i(\gamma)) = d\cdot l_\hyp (\gamma),
$$
and the conclusion follows. 
\end{proof}

\sss{Flux vs weight}

\begin{lem} \label{flux vs weight}
  Let $\alpha $ be an arc, and let $\si$ be a segment with both ends on
  $\alpha$ that land on the opposite sides of $\alpha$ (the
  orientation preserving situation).  Let us consider a simple closed curve $\gamma$ obtained 
  by connecting the endpoints of $\si$ along $\alpha$. Then
$$
   \WW(\gamma) \geq \frac 12  \bar\WW (\si), 
$$  
as long as $\bar \WW(\si) > 2/\pi$.

\end{lem}

\begin{proof}
  Let $\alpha$ lenad on boundary components $I$ and $J$.
Then our configuration of arcs lifts to the universal covering $\D$ as
follows:

\begin{itemize}
\item $I$ lifts to two ideal arcs $\BI$ and $\BI'$; 
          $J$ lifts to two ideal arcs $\BJ$ and $\BJ'$

\item $\alpha$ lifts to two arcs: $\Bal$ connecting $\BI$ to $\BJ$,
  and $\Bal'$ connecting $\BI'$ to $\BJ'$. 

\item $\si$ lifts to an arc $\Bsi$ connecting $\Bal$ and $\Bal'$;

\item $\gamma$ lifts to aa axis $\Bga$ connecting the ideal gap $\BG$
  between $\BI$ and $\BJ$ to the ideal gap $\BG'$ between $\BI'$ and $\BJ'$
  (and separating $\BI$, $\BI'$ from $\BJ$, $\BJ'$.

\end{itemize} 

Let us uniformalize $\D$ be the horizontal strip of width $\pi$ so
that $\Bga$ becomes an axis of this strip. (We will keep the same  notatin 
for all the objects under consideration.) The deck transformation
$T_\gamma$
along this axis is a translation by some number $a$. Moreover,
$\WW(\gamma) = \pi/a$. 

The curves representing $\WW(\si)$ connect $\BG$ to $\BG'$,
so they go over the convex hulls  $[\BI, \BI']$ and $[\BJ, \BJ']$,
both of which have length greater than $a$. Hence these curves have
length at least $\min (\pi, a)$. 

Let us put the Euclidean metric on the union of two rectangles based upon the above
intervals
and going across the strip. Then the length of all curves with respect to this
metric   is bounded in the same way. The area of these metrics is
bounded by $2\pi a$. Hence   
$$
  \bar\WW(\si) \leq \frac {2\pi a} { \min (\pi, a)^2  } =\max \left(\frac{2\pi}
  a, \frac 2\pi \right) = 2 \max ( \WW(\gamma), 1/\pi). 
$$
The conclusion follows. 
\end{proof}

******* } 

\comm{***** 

\sss{Flux through the ray diagram}

For an arc $\alpha$ on $U'\sm \KK'$, the inclusion $i_*(\alpha)$ is well
defined if

\ssk\nin $\bullet$
$\alpha$ is horizontal with both ends landing on some $K_n$; 

 \ssk\nin $\bullet$
or $\alpha$ is vertical with one end landing on some $K_n$; in which case it admits a unique extension
to a vertcal arc $ i_*(\alpha)$ on $U$ by means of the homotopy
retraction $U\ra U'$. 

\msk
We say that a WAD $X'$ on $U'$ is {\it invariant} under $F$  if :

\ssk \nin $\bullet$ 
The inclusions $i_*(\alpha)$ are well defined for all $\alpha\in
\supp X'$; 

\ssk \nin $\bullet$ $f_*\alpha$ is an arc for any $\alpha\in \supp X'$
  (in other wirds $\alpha$ and $-\alpha$ do not cross); 

\ssk \nin $\bullet$   $i_*(X')= f_*(X')$.

\ssk\nin
We let $X=f_*(X')=i_*(X')$ be the image WAD on $U$. 

\msk
For instance, consider the {\it ray diagram}  $\RR'$ on $U'\sm \KK'$ supported on all external rays
landing on the little Julia sets $K_n$, $n =0,\dots, p-1$, with weights
$$
   r_n' =  (\# \ \mbox{ of external rays landing on $K_n )^{-1}$. }
$$

\begin{lem}\label{flux through inv WAD}
    Let $X'=\RR'$ be the  invariant ray daigram on 
$U'\sm f^{-1}    (\KK)$. 
Then
$$
  \WW^\perp (Y')\geq \WW^\perp (X') - O( p (1 +  \WW^\ver) ) 
$$ 
for any   WAD $Y'$  such that  $i_* (Y') = i_* (X')$.
\end{lem}

\begin{proof}
Let  us decompose $i$ as $j\circ e$ where
$$
  e : U'\sm (\KK\cup \KK') \ra U'\sm \KK,\quad j: U' \sm \KK \ra U\sm \KK.
$$
We let $Z = e_*(X')$, so $j_*(Z) = X$.   Then
\begin{equation}\label{restricted X}
   \WW^\perp (X) \leq \WW^\perp (Z) + O(p (1 +  \WW^\ver) ). 
\end{equation}
Indeed,  for any proper path $\gamma$  in $U\sm \KK $  crossing $\alpha \in \supp X$, 
a lift $j^*(\gamma)$ that begins on a little Julia setis either
crosses $ j^*\alpha\in Z$  or is vertical
 (and the vertical weights
get summed up over the whole diagram $X$ containing $O( p) $ arcs). 

\comm{******
If it is vertical, three situations can occur:

\ssk\nin (i) $\gamma$ is well localized in the translation region
(meaning that its ``dives" into the Translation region withon a
bounded number of blocks). The total weight of such  

\ssk \nin (ii) $\gamma$ dives into the Translation region in a
non-localized way. The total weight of such arcs is
$O(p+\WW_{periph})$. 

\ssk \nin (iii) $\gamma$ exits the Translation region.
The total weight of such arcs is
$O(\WW_{periph})$. 

*********}

Since $i_*(Y') =X$, we have $e_*(Y') = Z$, and Corollary 
\ref{flux through arc diag} yields: 
$$
   \WW^\perp (Y') \geq \WW^\perp (Z) - O(p).
$$
Putting it together with (\ref{restricted X}), we obtain: 
$$
      \WW^\perp (Y') \geq \WW^\perp (X) - O(p(1+  \WW^\ver)) 
$$
But since $f: U' \sm (f^{-1}\KK)\ra U\sm \KK$ is a covering map and
$f_* (X')= X$, 
we also have:
$$\WW^\perp (X') = \WW^\perp (X).$$ 
\end{proof}
   
\begin{rem}\label{refinement}
The above estimate comes through if we replace the Riemann surfaces
$S=U\sm \KK$ and $S'= U'\sm f^{-1} \KK$ with 
$$
    S^n  = U^n\sm f^{-n} \KK. \quad \mathrm{and} \quad   S^{n+1}   
$$
coming together
with the  covering-immersion pair $F =(f ,i) : S^{n+1} \ra S^n$, 
and two   arc diagrams,  $X^{n+1} =\RR^{n+1} $ and $Y^{n+1}$,  on $S^{n+1}$ 
such that $X^{n+1} $ is $F_n$-invariant and  $i_ * (Y^{n+1} ) = i_*(X^{n+1})$.   
The error term should be interpreted as the vertical weight on $S^{n+1}$. 
When we want to emphasize  where the weight is calculated, 
we use notation like $\WW(X; S)$. 
\end{rem}

Next, we will sharpen the error term in  the previous lemma.

\begin{lem}\label{flux through inv WAD-improved}
Fix an arbitrary small $\eps>0$ and a big $N$. 
Under the circumstances of the previous lemma
(adjusted as in the above Remark), 
either $W^\ver \geq N \WW^\loc$
or else
$$
  \WW^\perp (Y^n;\, S^n  )\geq 
\WW^\perp (X^n; \,    S^n) - \eps\, p\, \WW^\loc  
$$ 
for some $n = O( N/\eps) $. 
\end{lem}  

\begin{proof}
We will pursue the following strategy. 
  If for some restriction to  $S^n$, 
the newborn vertical saisfies $W^\ver (S^n) < \eps \WW^\loc$, 
then we obtain the desired estimate. Otherwise, we have the
definite (in terms of $\WW^\loc$) vertical weight created under each restriction. 
Under $\asymp N/\eps$ steps, we will create at least $N\WW^\loc$ of
the vertical weight.
Then we will push it forward by a baby version of the Covering Lemma.

Let  
$$
\Si^n : = U^n \sm f^{-(n-1) } \KK.
$$


Let us consider the covering annulus $\A^n$ of $S^n$ corresponding to the loop
$\di U^n $, and let $\BBB^n $ be the similar annuli for the $\Si^n $.

Let  $\Pi^n$ be the universal coverings of the $\A^n$  realized as strips
(with the top $T^n  =\R+ h_n \pi $  covering the outer boundary $\di U^n$ and the
bottom $\R$).  We  select the heights $h_n$ so that $\A^n= \Pi^n/\Z$. 
Similarly, $P^n$ are the univesal coverins of the $\BB^n $. 

We mark external rays  $\hat \RR^n(i)$  on the annuli
and the strips within one fundamental region
(using the same notation in $\A^n$ and $\Pi^n$). 
They select for us  initervals $\hat K^n(i)$ on the bottom $\R$ corresponding to the little Julia
sets. By definition, the flux $\WW^\perp$ through $\RR^n(i)$ is
equal to the height of 
the quadralateral $\Pi^n(i)$ whose horizontal sides are
the top of $\Pi^n$  and the interval $\hat K^n(i)$. 

\comm{*****
Let $I^n$ be the interval on the bottom $\R$ of $\Pi^n$ which is the
convex hull of the  the intervals $\hat K^n ( q-(n-1) ) $ and 
$\hat K^n (q- (n-3))$. 
Let $w^n $ be the  the width of the path family going from $\R$ to
$\R$ skipping over $I^n $
(i.e., the height 
 of the rectangle $\Pi^n$ with
horizontal sides $I^n$ and $T^n$). 
We claim that the $w^n$ are small.

To see it, let us also consider the
covering map $f_n :  \A^n \ra \A$ of degree $2^N$ and its linear
lift $\hat f_n :  \Pi_N\ra \Pi$. 
Under this map,  the interval $\hat f_n  ( I^n) $ 
covers (essentially with degree one) the whole bottom of $\A$, which implies that $w_n$ is bounded
by the height of  $\A$ plus $O(1)$ 
Since this annulus is bad, we are done. \note{comment?} 
*****}

Let $k=[1/\eps]$, and let $n\leq kN$. 
Let us consider the foliation $\FF^n (i)$ representing the flux
through $\RR^n (i)$ in $\Pi^n $. 
Let us pull it back by $i^*$ to $\Pi^{n+1}$.  
A leaf $L$ can either be fully lifted to $\Pi^{n+1}$ or else it can be broken.
 It the former case,  it contributes to the flux through $\RR^{n+1}(i)$.
It the latter case, it has at least  two vertical lifts that begin on
the opposite side of $K^n(i)$.  

So, the weight of the broken leaves is bounded by the vertical weight
of $A^{n+1}$.  If the total of this flux  is smaller then $\eps p  \, \WW^\loc$ then 
we obtain  the desired error estimate. 

Otherwise, for all  $n\leq k N$, the weight of broken leaves of
$\FF^n(i)$ (for any $i$ in the ``middle''  the translation region) is at
least $\eps\WW_\loc$. 
After $kN$ restrictions we create at least $N\WW_\loc$ of the vertical
weight.  

If it is localized outside the interval $I^n_{kN}$ in $\A^{kN}$ contianing the
last $kN$ sectors, 
then there is  a standard rectangle $P$ in the translation
region but outside $I^n_{kN}$ whose width is at least as big  (up to a bounded additive
error).  Under the map 
$\hat f_{kN}: \Pi^{kN} \ra \Pi$ composed with the projection $\Pi\ra \A$, 
this rectangle is mapped with degree one to $\A$.
Hence  $A$ has at least as big width.  
%

Otherwise we would create a big vertical weight in $\A^n$ landing on
$I^n_{kN}$.  Let us consider a disk $D$ containg  $kdN/2$ little Julia
sets (in the $kN/2$ consecutive sectors) aligned with the Hubbard
marking.  Let $\mathbf D$ be the corresponding psuedo-disk. 
The above vertical weight blocks the canonical arc diagram in $\mathbf
D$ to skip over several little Julia sets in the same translation
orbit from ``above''. 
 Neither an arccan be a ``serpant'' winding above these little Julia
 sets
since it would create a self-intersection. So, any arc can leave
$\mathbf D$ only through $\RR(0)$ or  by going ``under'' all little Julia sets,

In either case, its push-forward
would  intersect itself. 
\end{proof}

\sss{Cobras}

Let us begin with a non-dynamical situation. 
Let us mark an (oriented)  arc $\alpha$  in $U\sm \KK$ connecting a
boundary component  $I$ to a component $J$.  
A  segment $\si$ in $(U, \alpha)$  is called a {\it cobra}
(attached to  $\alpha$) if it  begins and  ends  on $\alpha$ and is
homotopically trivial in $U\sm (I \cup J \cup \alpha)$ (where its endpoints
are allowed to slide along $\alpha$).  
It is oriented so that its end is ahead of the tail (on $\alpha$).  
 Its weight
$\WW(\si)$ is the flux through $\alpha$ along $\si$.  \note{defined?}

 Let us consider a twisted arc $\alpha^\si$ on $U\sm \KK$  obtained by adding to $\alpha$
 the  segment $\si$, while erasing the corresponding piece of $\alpha$.  \note{picture}
The following lemma shows that this operation leads to a loss of the
flux (as long as the weight of the cobra is sufficiently big). 

\begin{lem}\label{loss of cobra}
If $\WW(\si) > 0 $ then 
$$
 \WW^\perp (\alpha^\si)  \leq  \WW^\perp (\alpha) - \WW( \si ) + O(1).
$$
\end{lem}

\begin{proof}
   Let us lif the segment $\alpha$ to a segment $\ba$ connecting some
   ideal intervals $\BI$ and $\BJ$ (lifts of of the corresponding
   boundary components). 
   Let $\Bsi$ be the lift of the cobra that begins on $\ba$. It
   connects $\ba$ to another lift $\ba_1$ of $\alpha$.  
  The latter begins on some ideal interval  $\BI_1$ (another lift of $I$)  and ends on some
  $\BJ_1$.

Assume for definiteness that the cobra $\si$ is negatively oriented
with respect to $\alpha$ (i.e.,  we turn right from $\alpha$ to
$\si$). Then the ideal intervals under consideration appear on the
circle $\T$ in the following cyclic order: $(\BI, \BJ_1, \BI_1, \BJ
)$. Let us consider the canonical arc-segment diagram of this partially marked
disk. By our assumption, it contains the arc $\Bsi$ connecting the gaps 
$(\BJ, \BI)$ and $(\BJ_1, \BI_1)$. Besides, it may contain four segments $\Bga(\BI), \dots,
\Bga(\BJ)$ connecting the gaps attached to the corresponding intervals
$\BI$ etc., and  some arc-segments begining on these
intervals (altogether there are at most five arc-segments with
positive weight).   

 Let us first assume that all the latter arc-segments have zero weight.
 Then we have:
$$
   \WW^\perp(\ba) = \WW(\Bsi) + \WW(\Bga(\BI)) + \WW(\Bga(\BJ)) + O(1),  
$$    
$$
   \WW^\perp(\ba^\si) = \WW(\Bga(\BI))  + \WW(\Bga(\BJ_1)) +O(1). 
$$      
So, it is enough to show that
\begin{equation}\label{desired estimate} 
   \WW(\Bga(\BJ_1)) \leq \WW(\Bga(\BJ) ).
\end{equation} 

Let $T$ be the deck transformation moving $\ba_1$ to $\ba$. 
It moves the configuration $\CC:=(\BI, \BJ_1, \BI_1, \BJ)$ to 
$T(\CC)= (\BJ, \BI,  T(\BJ),   T(\BI) ) $ (up to cyclic permutation).
Let $L$ be the gap  between $T(\BI)$ and $\BJ$, and let $\Bde$ be the
segment of $T(\CC)$  connecting $L$ to  the gap $(\BJ, \BI)$. 
Then 
$$
  \WW(\Bga(\BJ_1)) = \WW(\Bde).  
$$
Moreover,  the  fixed points $q^u$ and $q^s$ of $T$ lie in the gap $(\BI_1, \BJ)$,
(with $q^s$ ahead of $q^u$). \note{to be justified}
It follows that $T(\BI)$ lies in  between $q^s$ and $\BJ$,
and thus  $L\subset (\BI_1, \BJ)$.
Hence $\WW(\Bde) \leq \WW(\Bga(\BJ) )$, and 
the desired estimate (\ref{desired estimate}) follows. 

\msk In the general case, there could also be the arc-segment $\Bom$ connecting
$\BI$ to the gap $(\BJ_1, \BI_1)$ that intersects $\Bsi$ but not
$\ba$. 
But then $T(\Bom)$ goes from $T(\BI)$ to $(\BJ, \BI)$  contributing
the same flux through $\ba$.
\end{proof}

 Let us now consider  a family $\SSS$ of disjoint cobras $\si_i$ with
 the same (say, negaive) orientation attached to
 some arc $\alpha$. Let $D_i$ be the Jordan disk bounded by $\si_i$ and
 the segment of $\alpha$ sharing the endpoints with $\si_i$.   
Any pair of these disks are  either nested or disjoint,
so we can talk about  innermost and outermost cobras. 
The outemost cobras form a  family with disjoint $D_i$.
Let us define the twisted arc $\alpha^\SSS$ 
by twisting $\alpha$ simultaneously  along the outermost cobras. 

We define $\WW(S)$ as the total weigh of the cobras in $\SSS$. 

\begin{cor}\label{loss of several cobras}
Let  $\SSS$ be a family of negative cobras of positive weight attached to some arc $\alpha$. Then 
$$
 \WW^\perp (\alpha^\SSS)  \leq  \WW^\perp (\alpha) - \WW( \SSS ) + O(p).
$$
\end{cor}

\begin{proof}
  The twisted arc $\alpha^\SSS$ can also be obtained by twisting
  $\alpha$ consecutively along the innermost cobras. In this way we
  will make use of all the cobras of $\SSS$, 
loosing each time the weight of the cobra, up to a bounded additive
constant
(Lemma \ref{loss of cobra}). Since there exists at most $O(p)$
disjoint cobras, the conclusion follows.  
\end{proof}

More generally, we can consider a disjoint family $\SSS$  of  cobras
with the same orientation attached to
a WAD $X$, and twist $X$ along this family.  
It immediately follows that 
\begin{equation}\label{cobras attached to diag}
    \WW^\perp(X^\SSS) \leq \WW^\perp (X) - \WW(\SSS) + O(p), 
\end{equation}
where $\WW(\SSS)$ should be computed with the corresponding weights. 

Let us now pass to a dynamical setting. 
Let $X$ be an invariant WAD in $U'\sm (\KK \cup \KK')$.
We say that the cobra $\si$  attched to $X$  is {\it tame} if $i_*(\si)$ is a trivial
segment on $U'\sm \KK$
(so, $\si$  surrounds only isles of $\KK'$). 

\begin{lem}\label{tame cobras}
Fix a big $N$.
Let $X$ be an invariant WAD, 
and let $ \SSS $  be a family  of disjoint  tame cobras $\si_i$
attached to $X$.  
If $\WW^\loc$ is sufficiently big then either $\WW^\ver\geq N \WW^\loc$ or else
$$
      \WW (\SSS)   = o ( p \WW^\ver ).   
$$
\end{lem} 

\begin{proof}
We can consider separately negatively oriented and positively oriented
families. In either case,  the twisted WAD $X^\SSS$  is  well defined. 
By (\ref{cobras attached to diag}),  
 we loose the total (weighted) flux of these cobras.
On the other hand, since the cobras are tame, we have $i_*(X^\SSS) =
  i_*(X)$, so by Lemma \ref{flux through inv WAD-improved}
$$
     \WW^\perp (X^\SSS) \geq \WW^\perp (X) -  o ( p  \WW^\loc).
$$
and the conclusion follows. 
\end{proof}

\comm{****
\sss{Generalized cobras} 
We will need  more general creatures:
a {\it generalized cobra} (attached to $X'$)  is a segment $\si'$ in
$U'$  with endpoints on two adjacent edges $\RR_1$ and $\RR_2$  of
$X'$ (so, both $\RR_i$ land on the same little Julia set $J_n$).
A generalized cobra is tame if in the region bounded by $\si \cup
\RR_1\cup \RR_2 \cup J_n$ (where $J_n$ is understood as a piece of
the ideal boundary  of $U'$) 
there are no other little Julia sets $J_k$.   
Lemma~\ref{tame cobras} has the following generalized version': 

\begin{lem}\label{tame generalized cobras}
Let $X'$ be an invariant WAD, 
and let $ \SSS $  be a family  of disjoint  tame generalized cobras $\si_i$
attached to $X$.  Then
$$
      \WW (\SSS)   = O( |X'| + p\, \WW^\ver + |\SSS|\, \WW_\loc ) .  
$$
\end{lem} 

\begin{proof}
Proceed as in  Lemma~\ref{tame cobras}: 
 consider a WAD $Y'$ on $U'$  obtained by adding to $X'$
   segments $\si_i'$, while erasing the corresponding pieces of $X'$.  
   Then:
$$
 \WW^\perp (Y')  = \WW^\perp (X') - \WW({\SSS} )  
$$

  On the other hand, since the cobras are tame, we have $i_*(Y') =
  i_*(X')$, so by Lemma \ref{flux through inv WAD}
$$
     \WW^\perp (Y') \geq \WW^\perp (X') -  O( |X'|+ p\, \WW^\ver).  
$$
Comparing this with the previous formula, we obtain the desired.
\end{proof}

****}

\sss{Cobras in the Elephant eye}

Assume now that we have a multi-elephant combinatorics with $d$ eyes
(so the period is equal to $pd+1$. 
For the WAD $\RR'$ of external rays,
we let $\RR' (n)$ be the union of rays landing on $J_n$, and $\RR_k' (n)$
be the corresponding  individual rays.
Similarly,  for the diagram $\RR=i_*(\RR')$,
we use notation $\RR(n)$ and $\RR_k (n)$. 

Note that  $\RR (0)$ comprises two rays, so together with $J_0$ they
divide $U$  into two half-disk. Let $H_+$ be the one that does
not contain  the critical value $c$. For our combinatorics, it contains only
bounded number of the little Julia sets. Accordingly, $H_-$ is the
other half-disk. 


\begin{lem}\label{many tame cobras} 
 Any cobra $\si$ in $U \sm f^{-1} (\KK)$ attached to  $\RR(0)$ produces a tame cobra
 of the same weight in $U' \sm f^{-2} (\KK)$ attached to $\RR'(d\, p)$. 
\end{lem}

\begin{proof}
   Since the cobras in $U\sm f^{-1}(\KK)$ attached to $\RR'(0)$ form a $0$-symmetric 
family, we can assume without loss of generality that $\si$ is
contained in $H_+'$. 

    We have: $f(H_+')\supset H_+$,  
and all little Julia sets $J_{pk}$ in $H_+'$ except $J_{pd}$ are mapped
to $H_-$. It follows that the  lift of $\si$ 
 by $f:  H_+' \ra U $  is a  tame cobra in $H_+'\sm f^{-2}(\KK)$ attached to
$\RR'(d\, p)$. 
\end{proof}


\begin{lem}
   If the total weight of cobras in $U\sm \KK$ attached to $\RR(0)$ 
is at least  $p \WW_\loc$ (in order), we are done. 
\end{lem}

\begin{proof}
%
    Lemma \ref{many tame cobras} produces for us a family of tame
  cobras in $U'\sm f^{-1}(\KK)$ of the total weight $p\WW^\loc$,
  \note{be careful}
and Lemma \ref{tame cobras} leads to a contradiction. 
\end{proof}



\comm{****
\begin{lem}\label{many cobras}
Assume the canonical WAD is not localized in the translation region.    
At least $\asymp p\, \WW_\loc$  of the flux through $\RR(0)$ is supported on
 cobras in $H_+$.
\end{lem}

\begin{proof}
  In the non-localized case,  at least $\asymp p\, \WW_\loc$  of the flux through $\RR(0)$ 
is supported on the canonical WAD.   \note{justify}
Let us take an arc $\alpha$ of the
canonical WAD crossing $\RR(0)$. There are three possible cases:

\ssk \nin (i) It is trapped inside $H_+$;

\ssk \nin (ii) It returns to $\RR(0$ from $H_+$ crossing it through
the same ray as it entered;

\ssk \nin (iii)  
It returns to $\RR(0) $ from $H_+$ crossing it through
the diffenrent  ray.

\msk In case (i)  it lands on some  little Julia set in $H_+$
(possibly, on $J_0$), and the
total weihgt of such arcs is bounded by $(d+1)\, \WW_\loc$.  

\ssk
In case (ii) $\alpha$ bounds a cobra in $H_+$.

\ssk In case (iii), we have similar three subcases according to its behavior in $H_-$:

\ssk\nin (iii-a) After entering $H_-$ it is trapped over there; 

\ssk\nin (iii-b) It returns back to $\RR(0) $ from $H_- $ crossing it through
the same ray as it entered;

\ssk\nin (iii-c) It returns back to $\RR(0) $ from $H_- $ crossing it through
the different ray;

\ssk
Case (iii-a) leads to localization. \note{justify}

\ssk 
Case (iii-b) creates a cobra.

Finlly, in case (iii-c) we have again three subcases, etc. 

\msk In the end, we have: 

\ssk \nin 1)
 $\alpha$ either lands on some little Julia set in $H_+$:
totall weight of these arcs is still at most $(d+1) \, \WW_\loc$; 

\ssk \nin 2)  or it
creates a cobra based on $\RR(0)$ (which is what we want);

\ssk \nin 3) or it lands on some little Julia set in $H_-$ after a
bounded number of revolutions, which again leads to the localized case; 

\ssk \nin 4) or or it lands on some little Julia set in $H_-$ after an
unbounded number of revolutions.  
\end{proof}

\begin{lem}\label{many tame cobras} 
  Assume the canonical WAD is not localized in the translation region.   
Then  at least 1/2 of the flux through $\RR'(d\, p)$ is supported  on tame
    cobras. 
\end{lem}

\begin{proof}
    We have: $f(H_+)\supset H_+$, \note{for genuine ql maps}
and all little Julia sets $J_{pk}$ in $H_+$ except $J_{pd}$ are mapped
to $H_-$. Lifting by $f:  H_+'\ra U$  the cobras from  Lemma \ref{many cobras},  
 we obtain tame cobras in $H_+'$ attached to
$\RR(d\, p)$ that occupy at least 1/2 of the flux through $\RR'(d\, p)$. 
\end{proof}
****}

\comm{***********
\subsection{Rotation number} 

Let us define the {\it rotaion number}  $\rho(K)$ around a little  Julia set
$K=K_i$ as follows.   Take any ray $\RR$ landing at  $K$ oriented
from $K$ to $\di U$.   Consider the immersed foliation $\FF$
through $\RR$ oriented positively with repect to $\RR$. 
Then we have the (partially defined)   first return map $T: \RR\ra
\RR$. 
 For any $a\in \RR$ and $b= T(a)$, we let $c(a, b)= 1=
-c(b,a)$. If $T(a)$ is not well defined, we let $c(a,b)=0$ for all
$b\in \RR$.  
This function can be extended to a cocycle $c(a,b)$ for the
equivalence relation generated by $\FF$ on $\RR$. 
 Then we let
$$
   \rho(K_i) =  \limsup_{n\to \infty}   \frac 1n \int_{\RR} c(a,a_n) \, d\mu(a),
$$ 
where $a_n$ is the $n$th return of the leaf $L_a$ to $\RR$, 
and  $\mu$ is the transverse measure on $\FF$.

\begin{lem}
  The rotation number $\rho(K)$ does not depend on the choice of the ray
  $\RR$. 
\end{lem}
***************}

\subsection{Poincar\'e return map}

Given some arc $G$ for a map $f$, we can consider the return map $T=T_G$ to $G$ along the
immersed foliation $\FF^\perp$. 
More precisely, let us consider the covering corresponding to $U\sm
G$.  Its ideal boundary is marked with four ideal intervals, two of which, 
$\BG$ and $\BG^\#$, correspond to $G$.
Let us consider the canonical foliation on this surface. 
It defines the  transit map from $\BG$ to $\BG\sqcup \BG^\#$,
which  can be viewed as the return map to $G$.  \note{makes sense?}
 It comes together wirh 
the sign function $\eps(x)\in \{0,\pm 1\}$ recording whether the return is well defined
(otherwise $\eps=0$) and if so, whether it is  orientation
preserving (in this case, it goes from $\BG$ to $\BG^\#$) or
reversing (then it  goes from $\BG$ to itself). 

\begin{lem}
  Assume we have a sequence of maps $f_n$ as above. Then the
  corresponding return maps converge almost everywhere to a
   measure preserving automorphism. 
\end{lem}

Here the ``measure'' $\mu$  is a weak limit of the flux measures
$\mu_n$ through $G$. 

If $G=\RR$ is an external ray then  $T$ is monotonic on $\{x: \ \eps (x)=1 \} $,
since the corresponding foliation (albeit immersed) \note{OK?} 
 lives on the plane.

Given a sequence of maps $f_n$ with periods $p_n \to \infty$ 
and a ray  $\RR$ (independent of $n$ in an appropriate sense),
let us consider  the coresponding return maps $(T_n, \mu_n, \eps_n)$.

\begin{lem}
  Assume 
$$  \WW^\ver (f_n) \leq 2 \WW_\loc (f_n), \quad p\cdot \WW_\loc (f_n) = o(\WW^\perp (\RR, f_n)); \quad
  \WW^\perp (\RR, f_n)\asymp  \locflux.
$$ 
Then the maps $(T_n, \mu_n, \eps_n)$ have a limit
$(T, \mu, \eps)$. Moreover, $\eps(x)\not= 0$ a.e., $T$ preserves $\mu$ and is monotonic on 
 $\{x: \ \eps(x)=1 \} $.
\end{lem}

\begin{proof} 
Let us orient  the ray $\RR$ from the little $J$ -set to $\di U$.
For any point $a\in \RR$,  orient the leaf $L_a^n$ of
$\FF_n^\perp$ positively with repsect to $\RR$.  

Let us consider the set $X_n\subset \RR$ of points that
  never return to $\RR$ in positive time. 
This set can return to another edge $G$ of our cell decomposition,
but it must be   wandering  under the return maps.  Hence the number
of returns is bounded by   
$$
       (1+\de)  \locflux) /\mu_n(X)
$$
(since the  trnasverse measure is almost preserved  under the
holonomies). 
It follows that there exists at most 
$$
   C p\cdot  \locflux  /\mu_n(X)
$$
different homtopy classes of arc-segments in this foliation (with some
absolute $C$). Since each of them has weight at most $\WW_\loc(f_n) $, we
obtain:
$$
  \mu_n(X)  \leq Cp\WW_\loc (f_n)\cdot  \locflux
  /\mu_n(X),
$$
and thus
$$
   \mu_n(X) \leq \sqrt{Cp\WW_\loc (f_n)\cdot 
      \locflux } = o(\WW^\perp(\RR, f_n )). 
$$
Thus, in the limit the return map is well defined almost everywhere. 
\end{proof}

\begin{lem}
  Under the assumptions of the previous lemma,  one of the following
  situations occur:

\ssk \nin $\mathrm{ (i)}$ There is a cobra in $H_+$ of definite weight based upon  a
ray $\RR_i(0)$;

\ssk \nin $\mathrm {(ii)}  $  $T|\RR(1) =\id$ (so all the leaves of the ``limiting lamination'' are closed). 
\end{lem}

\begin{proof}
  If there is a cobra of definite weight based upon $\RR(1)$ then
  one of its lifts is the desired cobra based upon $\RR(0)$.  

Otherwise $\eps(x)=1$ a.e. for the return map $T_1$  to $\RR(1)$, so it is
monotonic. Since it is measure preserving, it is identical. 
\end{proof} 

\msk
\subsubsection{Quadratic-like restrictions}
Below, we will restric ourselves to the case of two eyes
(of period $2p+1$). 

Let us consider a cobra segment $S$ based on sone $\RR(n)$ with $n<p$
that begins by crossing  the edge $e_{n-1} $  connecting $K_n$ to $K_{n-1}$.
Assume it is part of some canonical arc $\alpha$.  Then typically it cannot cross the
path $\de_{n-2}$ of edges $e_{j-1}$ consecutively connecting $K_j$ to $K_{j-1}$,  $j=n-2,\dots
1$, for it would lead to an interection between $\alpha$ and $f^* \alpha$. Hence $S$
goes around  $\de_{n-2}$  returning to $\RR(n)$ on the other side of $\de_{n-2}$, or else
it comes back through $e_n$.  

\note{How about returning through $e_{n-1}$? }

There are only bounded number of segments of this kind. \note{why?}  
Assume the total weight of the corresponding arcs is of order
$p\WW_\loc$. Then we can push-forward such a segment to obtain at
most $O(p)$ segments of weight of order $\WW_\loc$. These segments
represent cobras based on $\RR_{n+1}$ of the same type, and there are
only a bounded number  of those,   \note{why?} 
so in fact, there is only bounded
number of push-forwards in question.  \note{elaborate!}

Doing this finitely many times,  we obtain a segment $S$  whose
push-forward $f_*S$ encloses $S$, creating a quadratic-like
restriction of the map.  \note{why?} 

\msk
Somewhat different approach. 
Let us  have a cobra $S$ (which is part of a canonical arc $\alpha$) 
 based upon some ray $\Ray(n)$ that together
bound a disk $D$. 
Assume:

\ssk \nin (0)
  $D$ does not contain $K_1 $; 

\ssk \nin (i)
 $\alpha$ is pushable forward 
(as we know, a typical arc is such);  

\ssk\nin (ii)
 $\Ray(n+1)$ is disjoint from $D$;

\ssk\nin (iii) 
There is  a maximal string of at least two consecutive  little Julia sets 
$K_l , \dots , K_m$ contained in $D$, with $m< n$. 

\ssk
Then $f_* (D)$ 
is a cobra-disk  based upon $\Ray(n+1)$ 
that encloses $K_{n+1}$. Hence it cannot be disjoint from $D$.
By (ii), $f_*(D)$ cannot be contained in $D$. 
Hence 

\ssk\nin
a) either $f_*(S)$ enters $D$ through the mouth of $S$ creating a new
cobra on $\Ray(n)$, or else  

\ssk \nin
b)  $f_*(D)\supset D$. 

\ssk\nin 
Assume in addition that

\ssk\nin (iv) 
 $D\supset K_0$

\ssk
 Then Case a) is excuded. Indeed $ f(D) \supset K_1$, 
while $f(D) \sm D \not\supset K_1$. \note{why?} 
Hence $D\supset K_1$ contradicting assumption (0).  

Case  b)  is excluded since otherwise we would obtain 
a  
quadratic-like map $f: D\ra f  (D) $.

\ssk 
How to excule the situaton when $D\not\supset K_0$ (to get rid of
(iv))? In this case, $D$ is sometimes a Thirston obstruction.
Namley, in case a) if $f(D)\sm D $ is

\msk
\subsubsection{Cobras with long tails} 

Let us consider a cobra based on $\RR(1)$ whose tail extends further
into the translation region crossing consecutively $\RR(1), \dots , \RR(N)$ so that the tail region
 bounded by $\RR(1)$ and $\RR(N)$ is empty (no little Julia sets inside). 
Let us pull it back starting with   the shift at the translation region. These
pullbacks contain a bounded number of little Julia sets located outside
the translation region.  Hence, after a bounded number of pullbacks,
we will either lose all these Julia sets, or will stabilize. In the
former case, we obtain an empty cobra that crosses about $N$ rays.
So, if the total weight of such cobras is of order $p\WW_\loc$, then the
total flux they produce is of order $Np \WW_\loc$.  This is not
allowed if $N$ is sufficiently big. 

In the latter case, we obtain a Thurston obstruction.   
 
\sss{Confined situation}
Assume we have cobras encircling  almost the whole set
of little Julia sets $J_{k+p}$, $0\leq k\leq p$. 
Then canonical arcs that start at $J_l$, $0\leq l\leq p$, 
either produce a quadrtic-like restriction, or empty cobras as above, 
or else they are confined.

\subsection{No closed curves}

\begin{lem}
  If $f$ is not intermediately renormalizable then the $2$-Canonical Closed Curve
  Diagram $\CC$ is empty. 
\end{lem}

\begin{proof}
First note that a curve $\gamma\in \CC$ cannot suround a single Julia set $J_s$ since in this case
the local weight of $J_s$ would be less than $1$.

 Let us take a curve $\gamma\in \CC$ surrounding the
  smallest number of little Julia sets.  
  Let $s\in [0,p)$ and $t\in (0,p)$ be the smallest ones such that $K_s$ and $K_{-t}$
  are surrounded by $\gamma_s:= \gamma$.
Let us pull $\gamma_s$ back $s+t$ times along the
  orbit of $K_s$. We obtain curves $\gamma_{s-1}, \dots, \gamma_{-t}$. 
By Lemma \ref{flux vs weight},
$$
   \WW(\gamma_{s-n}) \geq d_n^{-1}  \WW(\gamma_{s-n+1}),
$$
where $d_s=2$ and  $d_n=1 $ for $n\not=s$. It follows that all the curves $\gamma_{s-n}$,
belong to $\CC$.
Hence they are disjoint, so the correspoding Jordan disks $D_{s-n}$ are either disjoint or nested.  

Furthermore, since $\gamma$ surrounds the minimal number of little Julia sets,
all these curves 
surround the same number of little Julia sets.
It follows that if the disks $D_{s-n}$ and $D_{s-m}$ are nested then 
$\gamma_{s-n}$ and $\gamma_{s-m}$ are homotopic in $\C\sm \KK$.

Putting the above ingredients together, we see that $D_{-t}\cap \KK$
overlaps with $D\cap \KK$, the map $f^{s+t} : D_{-t}\ra D$ is a branched covering of degree two (if $s>0$) or one
(if $s=0$). 
Moreover, $\gamma_{-t}$ is homotopic to $\gamma$ in $\C\sm \KK$. 
In the former case,  $f$ is intermediately renormalizable;
in the latter case these curves form a Thurston obstruction (that does not exist in our situation).   
\end{proof}

\subsection{Long rectangles} 

\bignote{Written Feb 2017} 

We assume that the period is $2p+1$ (so it is really the doube eye case).

  Let us consider a non-confined arc $\alpha$ that begins and ends deep inside the
  translation region.  Then we have an initial and the final pieces of
  $\alpha$, $\alpha_i$ and $\alpha_f$. Both of them must exit the
  translation region.  Assume they  cross $\Ray (1)$ before exiting the
  region. (It should be possible to achieve  tihs  typically by taking
  $f_*(\alpha)$.) If there is another ray, $\Ray(n)$,  deep inside the tranlation
  region, then pieces of $\Ray(1)$ and $\Ray(n)$, together with pieces
  $A_i$ and $A_f$
  of  $\alpha_i $ and $\alpha_f$, form a long rectangle $R$. 
Let us analyze the structure of this rectangle. 

Assume first that both pieces $A_i $ and $A_f$  cross $\Ray(1)$
positively (``from inside'').  Let us also assume that we have just
one loop on the other side of $\Ray(1)$ and this loop does not contain
$K_0$, so this loop is a  is a cobra. 
Then consider several cases:

1) The rectangle $R$  is empty, i.e., it does not contain any little $K$-sets.
Then our cobra has a long tail.

2)  $R$ contains a string of consequtive little $K$-sets in the
adjunct orbit (of $K_{p+1}$). and then ``dive down " or ``up".
 Let us push forward $\alpha$. Then both horizontal sides of $R$
 go ``up" or ``down'', which forces $f(R)$ and $R$ cross. 

3) The pieces $\alpha_i$ and $\alpha_f$ don't dive, 
  but eventuallay land at the adjunct Julia set.
  Then look at the main string of little Julia sets (the translation orbit of $K_1$)
  and argue that they are confined.

  ***************  } 

\section{Elephant eye combinatorics}

We will use \cite{Book} for reference.

\sss{Description of the combinatorics}

Let $f: U' \ra U$ be an even  quadratic-like map with connected Julia set
$\Jul\equiv \Jul_f$ 
and both fixed points repelling.
Then one of these fixed point, called $\beta$, is the landing point of
the external $0$-ray,
while the other one, called $\alpha$, is the landing point of $\qq>1$
external rays $\Ray_i$ that a cyclically permuted with
some combinatorial rotation number $\pp/\qq$.
These rays partition $U$ into $\qq$ sectors $\Sec_i$ attached to $\alpha$,
counted anti-clockwise beginning with $S_0\ni 0$. 
See  \cite[\S 24.4]{Book}.

The symmetric rays $-\Ray_i$ land at the preimage $-\alpha$ of
$\alpha$.  Together with the rays $\Ray_i$, they partition $U$ into $2\qq-1$ domains: the
central one $\Ups$, the above $\qq-1$ sectors $\Sec_i$ attached to $\alpha$,
and symmetric sectors $\Zec_i$ attached to $-\alpha$.   
The intersections of these sets with $U'$ will be marked with
``prime'': $\Ups_i':= \Ups\cap U'$, etc.   

Assume that the map $f$ is renormalizable with period $p$ and little
(filled) Julia sets $K_n$, $n=0,1,\dots, p-1$, where $K_0\ni 0$
(see \cite[\S 28.4]{Book}).  We let $\Kfilled:= \bigcup K_n$.  

\msk
For the class of {\em elephant eye combinatorics} we assume:

\ssk\nin E1. The combinatorial rotation number of $\alpha$ is $1/\qq$.
  Then the domain $\Ups'$ is mapped onto $\Sec_1$ as a double covering,
  each $\Sec_i'$ and $\Zec_i'$ is univalently  mapped onto
  $\Sec_{i+1}$ for $i=1,\dots, \qq-2$, while each $\Sec_{\qq-1}$ and $\Zec_{\qq-1}$ is
  univalently mapped onto $\Ups \cup \bigcup \Zec_i\cup (-\Ray_i)$.

  \ssk\nin E2.
  $K_0$ is the only little Julia set contained in $\Ups$.

  \ssk\nin E3. $p= \qq+ \bb$, where $\bb$ is bounded by some $\bar \bb$
  (but $\qq$ is unbounded!).   We will also assume for simplicity that
    $\bb< \qq$. 

    \ssk Note that under these assumptions,
     each sector $S_i$, $i=1,\dots, \qq-\bb$,
    contains a single little Julia set, $K_i$.
    
  \sss{Hubbard tree}

  An efficient  way of describing the renormalization combinatorics is
  in terms of the Hubbard tree $\Hub\equiv \Hub_f$ (see \cite[]{Book}).
  Roughly speaking, it is the tree spanned by  the little Julia sets
  $K_i$ (collapsed to points) 
  inside the  big one.
  The vertices of the $\Hub$ are the collapsed $K_i$
(to which we will still refer as ``$K_i$''), together with the
$\alpha$-fixed point   and  all branched points.
(Note that for $\qq>2$, $\alpha$ is also a branched point.)%
\footnote{The subtlety of the precise definition  has to do with the
  possibility  that $\Kfilled  $ can be non-locally-connected.}

\comm{*******
  Let $\Hub_i$, $i\in \Z/\qq\Z $ be the limbs of $\Hub $ attached to
  the $\alpha$-fixed point. We also let

  \ssk\nin $\bullet$
  $[K_i,K_j]\subset \Hub$
 be the arc of the tree connecting $K_i$ to $K_j$
(more generally, we can consider the span $[T_1, T_2, \dots, T_k]$ 
of any family of subtrees);

  \ssk\nin $\bullet$
  $\Hub^\bullet\subset \Hub_0$ be  
  the span of  $-\alpha$ and  all little Julia sets contained in
  $\bigcup \Zec_i$;

  \ssk \nin $\bullet$
    $\Hub^\#\subset \Hub$ be the span of $\alpha$ and all $K_i$ with $i>\qq$ contained in
    $\bigcup \Sec_i$. 

  Under the Elephant eye combinatorics, we have: 

 \ssk \nin $\bullet$    
    for $i\not=0$,  $\Hub_i$ is   homeomorphically mapped onto $\Hub_{i+1}$;

\ssk\nin 
$\Hub^\bullet$ is  homeomorphically mapped onto
****************}

\bignote{An interesting combinatorial lemma is hidden here} 

\comm{*********
\begin{lem}
  For $0< n< p$, let $\gamma_n$ be the arc of $\Hub$ connecting
  $K_n$ to either $\alpha$ or $-\alpha$ (depending on whether
  $K_n\subset \bigcup \Sec_i$ or $K_n\subset \bigcup \Zec_i$). 
If $K_m\in \gamma_n $, $0< m<p$, then 
 under the Elephant eye assumption, $m>n$.  
\end{lem}

\begin{proof}
  Let us consider the maximal   $n$   for which the statement is
  violated.

  Assume first that $n< p-1$ and
  $K_n \not\subset  \Sec_{\qq-1}\cup \Zec_{\qq-1} $.
Then the arc $\gamma_n$ is homeomorphically mapped onto
$\gamma_{n+1}\ni K_{m+1}$,
with  $m+1< n+1< p$,
contradicting the maximality property of $n$. 

\bignote{OK not to  distinguish notation for a set $K_n$ from
  a vertex $K_n$?} 

If $n=p-1$, then the arc $\gamma_n$ is homeomorphically mapped onto
  the arc $[\alpha, K_0]$ containing $K_{m+1}$.
  On the other hand,  by Assumption E2,  $[\alpha, K_0]$ does not
  contain any little Julia sets inside. 

Finally, assume $n< p-1$ but $K_n \subset \Sec_{\qq-1} \cup \Zec_{\qq-1}$.
All the more,  $m< p-1$, so by Assumption E2,  $K_{m+1}$ must be contained
in some lateral sector $\Zec_j$.
Then the arc $\gamma_n$ is homeomorphically mapped onto an arc
$\de\subset \Hub$ connecting $\alpha$ to $-\alpha$,
entering $\Zec_j$, and then proceeding to
$K_{m+1}$ and further  to $K_n$, without ever exiting $\Zec_j$. 
It follows that its piece from $-\alpha$ to  $K_{n+1}$ is equal to $\gamma_{n+1}$,
and moreover, it contains  $ K_{m+1}$.
This is a  contradiction once again. 
\end{proof}
***************************}

For $ n\not= 0\ \mod \, p$, let $\gamma_n$ be the arc of $\Hub$ connecting
  $K_n$ to either $\alpha$ or $-\alpha$ (depending on whether
  $K_n\subset \bigcup \Sec_i$ or $K_n\subset \bigcup \Zec_i$).%
\footnote{We tacitly assume E2, though the discussion can be easily
 extended  to the general case.} 
We also let $\gamma_0^+$ be the arc of $\Hub$ connecting $K_0$ to
$\alpha$, while  $\gamma_0^-$ be the arc  connecting $K_0$ to
$-\alpha$. We let $\gamma_0$ be the concatenation of these two,
i.e., the arc connecting $\alpha$ to $\alpha_-$.
We have the following transformation scheme for these arcs:

\ssk\nin $\bullet$
   Each $\gamma_0^\pm$ is homeomorphically  mapped by $f$ onto $\gamma_1$;
  
   \ssk\nin $\bullet$
   If $0< n< p-1$ and $K_n \not\subset  \Sec_{\qq-1}\cup \Zec_{\qq-1}$,
   then  $\gamma_n $ is homeomorphically  mapped by $f$ onto $\gamma_{n+1}$;

\ssk\nin $\bullet$
    If $0< n< p-1$ and $K_n \subset  \Sec_{\qq-1}\cup \Zec_{\qq-1}$,
   then  $\gamma_n $ is homeomorphically  mapped by $f$ onto
   $\gamma_{n+1}\cup \gamma_0$; 

 If $n = p-1$ (and hence $K_n \subset  \Sec_{\qq-1}\cup \Zec_{\qq-1}$
   then  $\gamma_n $ is homeomorphically  mapped by $f$ onto
   $ \gamma_0^+$.

\begin{lem}\label{bounded deg}
  Under the Elephant eye assumption,
  any edge $\gamma$ of  $\Hub$ is mapped under $f^p$ at most two-to-one
 into the union of at most $\bb+ 1$ consecutive  edges.
\end{lem}

\begin{proof}
It is sufficient to show this for the arcs $\gamma_n$, $0< n < p $, and for
$\gamma_0^\pm$, which  follows from the following observations:

\ssk\nin $\bullet$
Each arc $\gamma_0^\pm$ and each arc $\gamma_n$,
$1\leq n  \leq  \qq-2 $,  is mapped  under $f^{\qq-n-1}$
homeomorphically onto $\gamma_{\qq-1}$;

\ssk\nin $\bullet$
By induction, for $\qq-1 \leq n \leq  p-1 $,  and $k\le \bb$,
the arc $\gamma_n$   is mapped at most two-to-one by
$f^k$  into $\gamma_{n+k}\cup  \displaystyle { \bigcup_{0 \leq i\leq
    k-1}\gamma_i }$, and therefore

\ssk\nin $\bullet$
For $\qq-1 \leq n \leq  p-1 $,   the arc $\gamma_n$   is mapped under
$f^{\bb} $ at most two-to-one  into $\displaystyle { \bigcup_{0 \leq i\leq   \bb}\gamma_i }$.
\end{proof}

\subsubsection{Translation domains}

\bignote{Define them here?}

Let us consider a topological disk $\Om $  inside the translation region
which is the union of $2N+1$ consecutive sectors
$\Sec_n$, $1\leq  l-N \leq  n\leq l+N\leq \qq-1$.
Ler $\Om'$ be its pullback of $\Om_l$ under $f^p$ based at $K_l$,
i.e., the connected component of $f^{-p} (\Om)$ containing $K$. 
Let us also consider the arcs $\de_n\subset \Om$
connecting $\KK_n$ to $\KK_{n+1}$ and aligned with the Hubbard tree
(that is,  homotopic to the  concatenation of the edges $\gamma_n$ and
$\gamma_{n+1}$ considered  above). 

\begin{lem}\label{J-sets don't pull off}
Under the above circumstances, $\Om'$ contains the little Julia sets $K_n$,
 $   l- N + \bb \leq n\leq l+ N-\bb'$, and the arcs $\de_n$ connecting
 them.  
\end{lem}

\begin{proof}
Assume inductively in $k=0,1,\dots$ that the assertion is true for all
$\KK_n$ on combinatorial distance at most  $k$ from $\KK_l$.
Let for definiteness $n=l+k \geq l$ (and $n<\qq- \bb$). 
By Lemma \ref{bounded deg},  the  image $\om_n= f^p(\de_n) $ is contained in the
union of arcs $\de_{n+j}$, $  |j|\leq \bb$, so it is contsined in
  $\Om$.  It follows that $\de_n$ is the lift of $\om_n\subset \Om$
  based at $\KK_n$. Hence $\de_n$ is contained in $\Om'$, and so is $\KK_{n+1}$.  
\end{proof}

\begin{lem}\label{interesting moments}
 Let $(f^n)^* \Om\not\supset K_1$  
 be a pullback  of $\Om$
 containing a horizontal arc $\alpha$
 with endpoints on combinatorial distance at
most $N-\bb$ from  the base set $\KK_{l-N}$.
Then   $f^* \alpha$  does not pull off.
\end{lem}
 
\begin{proof}
Let the arc $\alpha$ lands on $K_i$ and $K_j$. By Lemma \ref{J-sets
  don't pull off}, none of these  Julia sets get lost under the pullback
(i.e.,  both of them are contained in $(f^{n+1})^*\, \Om$).
Let $\beta= f^* \alpha$ be the lift of $\alpha$ that begins on $K_{i-1}$. 
If it pulls off then  it ends on $K_{j-1}'$,
implying$K_{j-1}'\subset ( f^{n+1})^*  \Om$. 

Since  $( f^n)^* \Om $ does not contain $K_1$, 
the map $f: ( f^{n+1})^*\, \Om \ra ( f^n)^*\,  \Om $ is univalent. 
It follows that $K_{j-1}\not\subset ( f^{n+1})^*\,  \Om $ -- contradiction. 
\end{proof}

 \section{Dynamical weights}

\sss{Local weights}
We will now consider a renormalizable $\psi$-ql map
$$  F\equiv (f,i): U' \ra U $$
with little  Julia sets $ J_n =\di K_n $, $n=0, 1, \dots, p-1$, and their
preimages $J_n' = \di K_n' = -J_n$, $n=1,\dots, p-1$.
(As usual, we identify $J_n/ K_n$  in $U'$ with $J_n/K_n$ in $U$.)

The following lemma shows that 
all the local weights $\WW_n^\KK = \WW (J_n) $ are $2-$comparable:

\begin{lem}\label{comparison of dyn loc w}
We have:
$$
       \WW_1^\KK \leq \dots \leq \WW_{p-1}^\KK  \leq \WW_0^\KK \leq 2\WW_1^\KK
$$
\end{lem}

\bignote{Prepare the ransformation Rule!} 

We let $\WW_\loc^\KK =  \WW_0^\KK \equiv \max \WW_k^\KK $.

The local weight $\WW(\di U)= (1/2) \, \WW(\di U') $
is called $\WW^\ver$.

\sss{Local fluxes}
%

For $0\leq n\leq p-1$, 
let us define {\em local flux}  $\WW_n^\Hub $
as the sum of all $\WW(K_m)$ and $\WW^\perp(G_i)$ comprising the
path $\gamma_n$ in the Hubbard tree $\Hub$.

\bignote{Need to extend the above discussion of fluxes to surfaces with
  punctures (at $\alpha$ and $\alpha'$).} 
  
The paths $\gamma_n$ form a Markov tiling of $\Hub$,
so we can consider  the corresponding $p\times p$  {\em Hubbard matrix} $A_\Hub$. 
Applying the Transformation Rules  (Lemma \ref{generaltransform rules})
and the Parallel Law (Lemma \ref{transform rule for concatenation}), we obtain

\begin{lem}
      $\displaystyle {\WW_n^\Hub \leq \sum_{A_{nm} =1} \WW_m^\Hub } $
\end{lem}  

Putting this together with Lemma \ref{bounded deg}, we conclude:

\begin{lem}\label{loc fluxes} 
  All local fluxes $W_n^\Hub  $ are comparable,
  with a constant  depending only on $\bb$.
Hence
$
     \WW(\SAS) \asymp p \WW_\loc^\Hub.  
$
\end{lem}

We let $\WW^\Hub_\loc : = \max \WW^\Hub_n$.
Note that by definition $\WW^\Hub_\loc \geq \WW^\KK_\loc$
but {\em a priori} they can be incomparable.    
We will see below that this is not the case in our situation. 

\msk

\bignote{Now $\WW^\Ray_n$ is put on the same footing.
Call them ``external'' fluxes? (while  the fluxes through $\Hub$
``internal''?
Use notation $\WW_n$, $\WW_n^{\mathrm in}$, and $\WW_n^\ext$?} 

\sss{Segment diagram weight} 

We call a segment $\si$  in the translation region
``self-intersecting''  if 
its translation $\si'$  by $-1$ crosses itself
(amounting to self-intersection on the quotient  cylinder.) 

\begin{lem}\label{self-intersecting segments} 
  The total weight of self-interseting segments in  the translation
  region is $O(p+\WW_\loc^\Hub)$.   \note{Check!!} 
\end{lem}

\begin{proof}
For a self-intersecting sgment $\si$, the translation $\si'$ cannot
represent the lift $f^*(\si)$, since otherwise tho segments in $U'\sm
\KK'$ would intersect. The alternative is that $f^*(\si)$ is broken.
By Lemma \ref{key estimate}, the total weight of these segments is
$O(p+\WW^\Hub_\loc)$.
\end{proof}

Let us call a segment $\si$ {\em peripheral} if its translation $\si'$
exits the translation region.   \note{define!} 
Since the total weight of these segments is  $O(\WW^\Hub_\loc)$, \note{OK?}
we conclude: 

\begin{cor}
  The total weight of self-intersecting and peripheral segments in the
  translation region is
  $O(p+\WW_\loc^\Hub)$. 
\end{cor}  

\bignote{The error  term used to contain $\WW^\perp(\RR_0)+\WW^\ver$.
Why?}

The same analysis applies to the complement of the translation region
\note{be more precise!} leading to the following conclusion:

\begin{cor}\label{long and peripheral} 
  The total weight of self-intersecting and peripheral segments   is
   $O(p+\WW_\loc^\Hub)$. 
\end{cor}

With this in hands, only short segments in the translation region are left unattended.

\begin{lem}
  The total weight of short segments is $O(p+\WW^\ver)$.
\end{lem}

\begin{proof}
Indeed, any short segment extends to either one of the segments listed
in Corollary \ref{long and peripheral},
or a loop surrounding a little Julia set, 
or  a vertical arc-segment,
or else to a ``snake'' composed of inner and outer short segments.
Note now that
loops around little Julia sets have small weight
(since they are dual to the local fluxes),
while snake's weights are bounded  by one (since they
intersect their translations). The conclusion follows. 
\end{proof}  

Putting these together, we obtain:

\begin{cor}\label{total segments weight}
  The total weight of the segment diagram in $U\sm \KK$ is
  bounded by  $O(p + \WW^\Hub_\loc + \WW^\ver)$.
\end{cor}

For $\la>1$, let us say that we  enjoy $\la$-{\em ampliphication} if
$ \WW^\ver \geq \la \WW^\KK_\loc$. 

\begin{cor}\label{total segments weight 2}
 For any $\la>1$,  either we enjoy $\la$-ampliphication 
  or else
  $$ \WW(\SSS_{U\sm \KK }) = O_\la (p + \WW^\Hub_\loc ). $$
\end{cor}

\sss{Arc-segment and long arc diagrams}

\begin{lem}\label{arc-segments and long arcs}
  The total weight of arc-segments and long%
\footnote{in both the translation region and its complement} 
arcs   in  $U\sm \KK$ 
  is  bounded by  the total weight of segments in $U\sm \KK$. 
\end{lem} 

\begin{proof}
For the same reason as for long segments in $U\sm \KK$,
 all arc-segments and long arcs get broken under the immersion 
$i:  U'\sm \KK \ra U\sm \KK $.  By   Lemma \ref{broken arc-segments},
their total weight is bounded by the total weight of segments in 
$U'\sm \KK$. But the latter is bounded by the total weight of segments
in $U\sm \KK$. 
\end{proof}

Combining this with Corollary \ref{total segments weight 2},
we obtain:

\begin{cor}\label{everything} 
  For any $\la>1$,  either we enjoy $\la$-ampliphication  or else
  $$ \WW(\SSS) + \WW(\AS) + \WW(\AAA^\lo)  + \WW(\AAA^\pe) = O_\la (p + \WW^\Hub_\loc ), $$
where all the diagrams are considered in $U\sm \KK$.   
\end{cor}  

\sss{Conclusions}

Since the total weight $\WW(\SAS)$ of segments, arcs, and arc-segments in $U\sm \KK$
is of order $p \WW_\loc^\Hub$ (by Lemma \ref{loc fluxes}), we conclude: 

\begin{cor}\label{small proportion}
   For any $\la>1$,  either we enjoy $\la$-ampliphication,  or else
 \begin{equation}\label{small proportion estimate}
   \WW(\SSS) + \WW(\AS) + \WW(\AAA^\lo) + \WW(\AAA^\pe)    = o (\WW(\SAS) )\
   {\mathrm{as}} \  p\to \infty, \ 
   \WW(\SAS)\to \infty,
\end{equation}    
where all the diagrams are considered in $U\sm \KK$.     
\end{cor}


In summary, we obtain:

\begin{thm}\label{short arcs}
  For any $\la>1$,  either we enjoy $\la$-ampliphication,  or else
  the elephant's share of the total weight $\WW(\SAS)$ is supprted on
  short arcs of the translation region%
\footnote{which is the same as ``short arcs in the complement of the
  translation region''}
  $($provided $p$ and $\WW(\SAS)$ are sufficiently big$)$. 
\end{thm}  

\begin{proof}
  What is unaccounted in the left hand side of
  (\ref{small proportion estimate}) is the weight of short arcs.
\end{proof}  

\begin{cor}\label{loc weights comparison} 
If the second option of the above Theorem is realized then 
  
\ssk\nin {\rm (i)} 
$\WW_\loc^\Hub\asymp \WW_\loc^\KK$. 

\ssk\nin {\rm (ii)}
For any $\kappa\in (0,1)$ and $\eps>0$,
if the period $p$ and the  local weight $\WW_\loc^\KK$ are sufficiently big
then for any $n \in [1,\kappa \qq] $
at least $ (1-\eps) $-proportion of the local weight $\WW_n^\KK$ is supported on
the two short arcs emanated from $K_n$.  
%
\end{cor}

\bignote{The error term in (ii) used to be additive $O(1)$.
It can be selected as $O(1/p)$.  }

\begin{proof}
(i)   Indeed, $\WW(\AAA^\sh) = O(p\, \WW_\loc^\KK)$
  while $\WW(\SAS) \asymp p\, \WW_\loc^\Hub $, and
  by   the last theorem, $\WW(\AAA^\sh) \sim \WW(\SAS) $. 

  \msk (ii) It follows from Corollary \ref{small proportion}
  that the assertion is true for some $n_0\in [\kappa \qq, \qq)$.
  Pulling the correspondng  short arcs back, we obtain the desired for all $n\leq n_0$. 
\end{proof}

%


\section{Pull-off Argument: bounded combinatorics}

Below $f: U'\ra U$ is a renormalizable $\psi$-quadratic like map, and
 $p$ is  its renormalization period. 
Let $U^n$ stands for the range of $f^n$ 
 (so $U^n= f^{-(n-1)}(U)$ in the genuine quadratic-like case). 
Let $\AAA$ be the canonical WAD on $U\sm \KK$ 
and $\AAA^n$ be the one on $U^n\sm \KK$. 

\subsection{Fast pull-off}

\begin{lem}\label{sym vert arcs} 
There exists a pair of symmetric vertical arcs in $U^p\sm f^{-1}\KK$ landing on $K_0$
that do not cross $\AAA^p$. 
Moreover, $f(\gamma^\pm)$ does not cross $\AAA^{p-1}$.  
\end{lem}

\begin{proof}

  There exists a vertical arc $\gamma_m$ in $U\sm \KK$  with some $0< m\leq p$ landing on $K_m$
that does not cross $\AAA$.
To see it, realize $\AAA$ as a planar graph, 
go from $\di U$ down to a small neighborhood of some arc $\alpha\in \AAA$,  
and then follow this arc in an appropriate direction until you land on some $K_m$. 

Pulling $\gamma_m$  back by $f^m$,
we obtain a  pair of symmetric vertical  arcs $\gamma^\pm$
in $U^m \sm \KK$ landing on $K_0$,
They do not cross  $\AAA^m$ since $f^*(\AAA)$ domaintes $\AAA$. 
 Restricting the domain  $U^m$ further to $U^p$, we obtain the desired pair of arcs. 
  
\end{proof}

\begin{lem}\label{in p steps}
Let $\alpha\in \AAA^p$ be a horizontal arc. 
Then it pulls off in less than $p$ iterated pullbacks. 
\end{lem}

\begin{proof}
  If $\alpha\in \AAA$ is a horizontal arc and $f^* \alpha$ is a legitimate pullback (i.e., with both ends landing on $\KK$)
if exists, then $f^\#\alpha:= i_* (f^*\alpha)\in \AAA$, where 
$$
    i: U'\sm f^{-1}(\KK) \ra U\sm \KK
$$
is the natural immersion. 
Indeed,  letting $\beta=f^* \alpha$, we have $ \WW(\beta) = \WW(\alpha) $  
since the covering 
$$
   f:  U' \sm f^{-1}(\KK) \ra U \sm \KK 
$$
induces an isomorphism between  the corresponding rectangles on the universal covering spaces. 
Furthermore, $\WW(i_* \beta)\geq  \WW(\beta)$ since the immersion $i$ induces an immersion between
the corresponding rectangles (on the universal covering space)
 that homeomorphically maps the horizontal sides of one to the
horizontal sides of the other. Thus, if $\WW(\alpha) > 2$ then $\WW(f^\# \alpha) >2$,
which characterizes the canonical arcs. 

Of course, the same applies to $\AAA^p$ instead of $\AAA$.  

Assume first we have a string of $3p$ legitimate pullbacks  $\alpha_m= (f^\#)^m \alpha$. 
Since $\AAA$ contains at most $3p-1$ arcs, two of these arcs are equal: $\alpha_m= \alpha_{m+n}$, $n>0$.
Then $\alpha_m =  (f^\#)^n \alpha_m$, which is impossible by Pilgrim's Lemma.   

So $\alpha$ must get pulled off in less than $3p$ steps. Assume now it pulls off in more than $p-1$ steps. 
Let $\alpha_m$, $m\geq p$,   be the last legitimate arc, and  let it land on $K_{i+1}$ and $K_{j+1}$. 
Then one of the arcs $f^* \alpha_m$ lands on $K_i$ \& $K_j'$ (and the other lands on  $K_j$ \& $K_i'$). 
Since $\alpha_m$ does not cross $f(\gamma^\pm)$, $f^*\alpha_m$ does not cross $\gamma^\pm$. 
Hence $K_i$ and $K_j'$ lie on the same side of $\Gamma:= \gamma^+\cup K_0\cup \gamma^-$,
and thus $K_i$ and $K_j$ lie on the opposite sides of it.  

But the arc $\alpha_{m-p}$ shares  the endpoints  with  $\alpha_m$, 
hence $\alpha_{m-p+1}$ (being legitimate) lands on $K_i$ and $K_j$. 
Hence it must cross $\Gamma$ -- contradiction.    
\end{proof}

\subsection{Newborn vertical weight}

\begin{lem}\label{def vert}
For any $p$, if $\WW_\loc$ is sufficiently big then 
the amount of newborn vertical weight in $U^{p+1}\sm \KK$ is at least 4\% of $\WW(K_1)$. 
\end{lem}

\begin{proof}
  Let $\WW_i=\WW(K_i)$ be the local weight around $K_i$ in $U^p\sm \KK$.
Each of these is  essentially localized on the CAD (up to an additive constant depending on $p$),
as long as $\WW_\loc$ is sufficiently big. 

Let $\nu$ be the total weight in $U^p\sm \KK$ of the horizontal arcs that are pulled off under $f^*$
(i.e., both their lifts are illegitimate).
By Lemma \ref{in p steps},  any horizontal arc $\alpha$ in $U^p$  landing on $K_0$ is pulled off in less than $p$ steps.
Let $\beta(\alpha) = (f^\#)^m \alpha$ be the last legitimate pullback of $\alpha$, $m<p-1$.

The correspondence  $\alpha\mapsto \beta$ is at most  two-to-one.
Indeed, if $\beta(\alpha)= \beta(\alpha')$ then $\alpha$ lands on $K_0$ and $K_l$ with some $l>0$,
while $\alpha'$ lands on $K_0$ and $K_{-l}$, and $(f^\#)^l \alpha =\alpha'$ (up to re-labeling the arcs).  
Hence we have:
$$
    \nu \geq \sum \beta(\alpha) \geq  \frac 12 \sum \alpha = \frac 12 \WW_0   -  O(1).    
$$
The total pulled-off weight available for composing new horizontal arcs on $U^{p+1}$   is $2 \nu + \WW_1$,
hence the new horizontal weight is bounded by  
$ (2\nu + W_1) /  4 $. 
So, by restricting $U^p$ to $U^{p+1}$, we have lost at least
\begin{equation}\label{hor loss}
  \nu -  \frac 14 (2\nu + \WW_1)  = \frac 12 \nu - \frac 14 \WW_1 \geq \frac 14 (\WW_0 - \WW_1) - O(1)
\end{equation}
of the horizontal weight. This creates as much of the vertical weight in $U^{p+1}\sm \KK$.  
If $\WW_0-\WW_1> \WW_1/5$, the last estimate implies the desired.

 Otherwise, we can use a rougher bound for the newborn  horizontal weight:
$$
     (\WW_0-\WW_1) + \frac 12 \nu \leq \frac 15 \WW_1 + \frac 12 \nu,
$$  
so the horizontal loss is at least
$$
  \nu-  (\frac 15 \WW_1 + \frac 12 \nu) = \frac 12 \nu - \frac 15 \WW_1 \geq \frac 14 \WW_1 -  \frac 15 \WW_1, 
$$
and we are done. 
\end{proof}

\begin{cor}\label{lots of vert}
For any $p$, if $\WW_\loc$ is sufficiently big then 
there is a definite  vertical weight in $U^{2p}\sm \KK$ (independently of $p$).
\end{cor}

\begin{proof}
  We can assume that the total horizontal weight in $U\sm \KK$ is of order $p\WW_\loc\sim \WW(\AAA)$
  (for otherwise there is nothing to prove). By the previous lemma
  (or rather, its proof) \note{need this?}  
by restricting  $U^{p+n}$ to $U^{p+n+1}$ we loose at least $c/p$-th part of the horizontal weight
(as long as it stays definite),
with an absolute $c$ (``4\%''). Hence by  $p$ restrictions we loose
at least $c/2$-th part \note{OK?}
of the total horizontal weight, thus creating as much of the vertical weight. 
\end{proof}

\comm{****
\subsection{Localization}

 We will now show that a definite (in scale of  $\WW_\loc$) newborn vertical weight is localized
at bounded number of little Julia sets.

Let us say that an arc $\alpha\in \AAA$ is {\it light } if $\WW(\alpha)< \xi\, \WW_1  (= 0.1\% \, \WW_1$); otherwise it is {\it heavy}. 
Since $\AAA$ is supported on a planar graph, there exist some vertex $K_s$ of valence $\leq 6$.
Light arcs at this vertex can weigh at most $5\xi \WW_s\hor+10 $.  If we pull the heavy  arcs around,
they get pulled off in at most $p$ iterates. Similarly to  (\ref{hor loss}),  
in this process we loose te least
$$
   \frac 14  ( (1-5\xi)\,  \WW_s  - \WW_1) - O(1)
$$
of  the horizontal weight. Moreover, this weight is localized at $\leq 6$ places. So, {\it  we win as long as $\WW_s$ is
definitely  bigger than $\WW_1$.}  
*****}

\section{Pull-off Argument: unbounded combinatorics}

\subsection{Setting}\label{assumptions sec} 

Consider the class $\QQ\equiv \QQ(\ul \qq, \bar\bb)$
of renormalizable pseudo-ql maps $f: U' \ra U $ with the elephant eye
combinatorics with parameters $ \qq\equiv \qq_f \geq  \ul\qq  $ and $\bb\equiv \bb_f\leq \bar \bb $,
and  renormalization periods $p\equiv p_f= \qq+\bb$.
Let $\AAA=\AAA_f$ stand for  the Canonical WAD of $f$  in $U\sm \KK$. 

Take an integer  $ N \in (1,     \qq-\bb]  $.
Let $\Om\subset U$ be  a topological disk 
in the translation region 
which is the union of 
$2N+1$ consecutive  sectors $\Sec_n\supset K_n$,
$n \in \II = [l-N, l+N] \subset \Z/ p\Z$.
Note that $\Om$ contains the chain of $2N$ short arcs
connecting the $K_n$, $n\in \II$
(or rather: appropriate representatives of these arcs). 
In what follows, we will also use notation $\De_m$ for $\Sec_m$. 

\comm{******
\ssk\nin (T1) 
$$
    \Om= \bigcup_{m=- N }^N \De_m,\quad \di \De_m\subset U\sm \KK,
$$
where the $\De_m$ are Jordan disks with disjoint interiors such that $f$ univalently maps $\De_m$ onto $\De_{m+1}$;
each $\De_m$ contains a block $\BB_m\subset \II$ of  $d$ little Julia sets, i.e.,
$$    
                        \De_m\cap \KK = \bigcup_{n\in\BB_m } K_n; 
$$


\ssk \nin (T2)  For any $n\in \II$, at least $\WW_n-O(1)$ part of the local weight $\WW_n$ is supported on $\AAA$
    (with an absolute additive constant);

\ssk\nin (T3) Any horizontal arc  $\alpha\in \AAA^\hor$ that begins in $\De_m$ with $|m| <  N$ 
  is confined to $\De_{m-1}\cup \De_m\cup \De_{m+1}$;

 ***************}

\bignote{In what follows, Ref (T2) should be replaced with Corollary \ref{loc weights comparison} (ii) }

 Such a domain $\Om$ will be called a {\it translation domain} of size
 $2N+1$ centered at $K_l$. 
Let us also   impose the following property on such a domain:

\ssk\nin (P1)
All local weights $\WW_n$, $n\in \II$,  are $(1-\eps)$-comparable.

\ssk\nin (P2)
$(1-\eps)$-proportion of the local weight $\WW_n$, $n\in \II$,
is carried by two short arcs emanated from $\KK_n$.

\ssk Here $\eps>0$ can be selected arbitrary small as long as
$\WW_\loc$ and $p$ are big enough.  

\ssk Note that (P1) is easily enforced because all the colca weights are
$2-$comparable (Lemma \ref{comparison of dyn loc w}). 
Property (P2) can be enforced to Corollary \ref{loc weights comparison}.

 We say that  $\Om\cap \KK$ is the {\it $\KK$-slice} of $\Om$.
Let  $(f^n)^*\Om $ be the pullback of $\Om$ under $f^n$ centered at $K_{l-n}$.
(We will see later that  actually  the pullback  $(f^n)^*\Om $
is independent of the base.) 

We let
$$
  \Om[\kappa]= \bigcup_{|m| \leq \kappa N} \De_m. 
$$

The Julia sets $\KK_{\pm N} $ will be called {\it peripheral}. Let
$\pp\II= \KK_N\cup \KK_{-N}$ be the family  of two peripheral Julia sets
and $\II^\circ= \II\sm \pp\II$ be the complementary family of non-peripheral Julia sets.

We let  $\BOM$ stand for the pseudo-disk corresponding to $\Om$, and  $(\Bf^n)^*\BOM$ for the corresponding  pullbacks.

For $s\in \II^\circ$, let $\CC(K_s |\, \II) $   be the conductance of the restricted electric circuit $\AAA|\, \II$ with battery $\{K_s , {\pp}\II\} $. 
Thus,  the potential difference 1 between $K_s$ and ${\pp}\II$ 
creates current  $\CC(\BB |\, \II) $ in the restricted circuit. 
By the Series Law, we have: 

If $s\in \KK_m$ with $|m|\leq  \kappa N$ for some
$\kappa\in (0,1)$ then
\begin{equation}\label{T4}
  \CC (K_s, \II  ) \leq \frac C{(1-\kappa )N} \, \WW_\loc,
\end{equation}
with an absolute $C$. 

\bignote{ Below Ref (T4) should be replaced with (\ref{T4}) ! }

It is convenient to play with many translation domains $\Om_l$
simultaneously:

\bsk\nin
{\bf Translation Property}.  \note{OK name?} 
For arbitrary big  $N, L\in \N$ 
there exists  $\underline{p}=\underline{p}(N,L)$ such that any 
map $f\in \QQ$ with $p \geq \underline{p}$ has  $L$ 
pairwise disjoint translation domains $\Om_l $
centered at some $K_l$
of size $2N+1$.

\msk
Here the index $l$ runs over some subset $\LL\subset \Z/p\Z$ of size $L$. 
All objects associated with $\Om_l$ will be labeled by $l$ if needed, e.g., 
$\II_l$ etc. \note{more?}  

\comm{*******
\subsection{Example: Multi-elephant eye with confined  canonical
  WAD} \label{confinement assumption}

Let us consider the multi-elephant eye combinatorics with $r$ eyes. 
Let $\Gamma$ be the almost invariant marking of $U\sm \KK$ by Hubbard arcs.

Given a subfamily  $\JJ\subset \Z/p\Z$ of little Julia sets,  
let us say that an arc $\alpha$ is {\it confined to $\JJ$} if it does not cross 
marked Hubbard edges  attached to any $K_m$ with  $m\not\in \JJ$. 

\bignote{Formulate in terms of sectors}

Let $\JJ_n = \{ n+jq \}_{j=0}^{r-1}$ be the ``fundamental family' 'of Julia sets
containing $K_n$, and let 
$$
   \JJ_n(R) = \bigcup_{ |k |\leq R}  ( \JJ_n +k ) 
$$ 
stand for its ``$R$-neighborhood''. 

For the moment, let us label the little Julia sets from $-p/2$ to $p/2$. 

\msk\nin {\bf Confinement Assumption:}
Any canonical arc that begins on  $K_n$ with $|n| \geq q/N$ 
is confined  to the region  $ \JJ_n(p/N)$. 

\bignote{Make it stronger: confinement to any region of size $q/N$
  with two buffers of the same size.}

\msk
Under this assumption,
 the canonical WAD well inside the translation region descends to the Ecalle-Voronin cylinder ${\Cyl}$.
Let us refine it to a triangulation of  ${\Cyl}$ with vertices $K_n$ and $\pm i \infty$.
Then we can connect $+i\infty$ to $-i\infty$ with a simple curve $\gamma$ composed of edges  of the triangles. 
 Cutting $\Cyl$ along $\gamma$ we obtain a topological strip in  $\Cyl=\C/\Z$ 
that lifts to a fundamental domain 
on $\C$.  
Taking the union of its $2N+1$ 
translates, we obtain $\Om$. 

********** }

\subsection{Amplification Theorem}

\begin{thm}
\label{amplification thm}
For any amplification factor $\la > 1$ 
there exist $\underline{\WW}_\loc = \underline{\WW}_\loc(N) $ and  $\underline{p}=\underline{p}(, \la, N,L) $
such that for all sufficiently big $N$ and $L$ 
we have:
$$
     \WW^\ver \geq \la \, \WW_\loc,
$$ 
as long as $\WW_\loc\geq \underline{\WW}_\loc$ and $p\geq \underline{p}$. 
\end{thm}

\subsection{Definite local loss}

Let us first show that if the local vertical weight is definite in  some pullback $(\Bf^n)^* \BOM$ with $n\leq 3p$,
then the Amplification Theorem holds:

\begin{lem}\label{def vert does it}
Assume there exists a  $\de > 0$ 
such that each domain  $\BOM_l$ has a  pullback $(\Bf^n)^*\BOM_l$ with $n\leq 3p$
that  satisfies
$$
      \WW ((\Bf^n)^*\BOM_l\sm K_{s-n}) \geq \de \, \WW_\loc, 
$$
where $K_{s-n}\in (\Bf^n)^* \BOM_l$ and the vertex $K_s$, $s\in\II_l$, 
stays distance at most $\kappa N$ from $K_l$. 
Then there exists  $L=L(N)$ such that the conclusion of the Amplification Theorem holds.
\end{lem}


\begin{proof}
We let $\BOM\equiv \BOM_l$ 
for simplicity. 

  Let us apply the Covering Lemma to  the map $\Bf^n: (\Bf^n)^* \BOM \sm K_{s-n}\ra \BOM\sm K_s$.
The local degree of this map is at most $8$; the big degree is at most $8^N$. The buffer is given by the maximal immersed annulus
in $\BOM\sm \KK$ around $K_s$; its width, $\WW_s \asymp  \WW_\loc$, is comparable with $\WW ((\Bf^n)^* \BOM \sm K_{s-n})$, by the assumption. 
Hence
$$
      \WW(\BOM\sm K_s) \asymp \WW((\Bf^n)^*\BOM\sm K_{s-n} ) \asymp \WW_\loc. 
$$

Any vertical path in
the annulus $\BOM\sm K_s$ either crosses  a peripheral
Julia set $K_m$,  $m\in \pp \II$, 
or  overflows a vertical path in $\BOM\sm\KK$ beginning on some  non-peripheral Julia set $K_m$, $m\in \II^\circ$. 

In the former case, our path follows the horizontal  electric circuit
in $\BOM$. By  (\ref{T4}),       
the width of this
family is bounded by $\frac C {(1-\kappa) N}\, \WW_\loc$ for $N$ big enough.  

The width of the latter  path family is bounded by 
$$
  \sum_{m\in \II^\circ} \WW^\ver (K_m|\, \BOM\sm \KK)  
                        = \sum_{m\in \II^\circ} \WW^\ver (K_m) +O(N),  
$$
where the last equality follows from  Corollary \ref{loc weights comparison} (ii).
where   $m\in \II^\circ$, $n\not\in  \II$).
Putting the above remarks together, we conclude that for some $n\in \II^\circ$, 
$$
   \WW^\ver (K_m) \geq \frac c {N} \WW_\loc.
$$

Let us now recall that by the Translation Property we can do it for $L$ domains $\BOM_l$, $l\in \LL$, with disjoint $\KK$-slices. 
This gives as $L$  different Julia sets $K_{m_l}$,  $m_l\in \II_l $. Letting $L= 2\la N/c $,
we obtain
$$  
   \WW^\ver\geq \sum_{l\in \LL}  (\WW^\ver (K_{m_l})- 2) \geq   \la \WW_\loc. 
$$
\end{proof}

\bignote{Section ``Essential arcs'' is hidden below}

\comm{**********

\subsection{Essential arcs}

For $\xi\in (0,1)$, 
we say that a horizontal arc $\alpha\in \AAA$ is $\xi$-{\it essential} if $\WW(\alpha)\geq \xi \WW_\loc$. 
We let $\AAA^\ess$ be the weighted arc diagram comprised of all essential arcs
(we will often let $\xi$ be implicit in the terminology and notation). 

For a simple closed curve $\gamma\subset U\sm \KK$, 
let the {\it flux} $\WW^\perp (\gamma)$ be the the width of the
annulus associated with $\gamma$ that covers $U\sm \KK$.

\begin{lem}\label{definite flux}
  Assume $p$ is the smallest renormalization period of $f$. 
 Then  for any natural $1\leq d<p $ there exists $\si =\si (d)>0 $ such that for any Jordan curve $\gamma$ surrounding $d$ little Julia sets
we have $\WW^\perp (\gamma) \geq \si \cdot \WW_\loc$.
\end{lem}

\bignote{Note that ``$d$'' is reserved for the number of Julia sets in the
  fundamental sector.}

\begin{proof}
   Let us do induction in $d$. For $d=1$ we can let $\si=1/2$. Let us pass from $d-1$ to $d$. 
Of course, we can assume that $\si(1)\geq \si(2)\geq\dots \geq \si(d-1)$.

Let $\gamma$ surrounds $d$ little Julia sets $K_n$, $n\in \BB\subset \Z/p\Z$, and let $\De$ be the disc bounded by $\gamma$.

\msk\nin
{\em Claim:} There exists $\xi=\xi(d)$ such that if $\WW^\perp(\gamma)< \xi/2$ then the graph of $\xi$-essential arcs in $\De$
is connected. 

\msk Since $\si(d-1)$ is well defined  inductively, we can let 
$$
   \xi(d)=  \frac {\si(d-1)} {3(d+1)}. 
$$
Assume $\BB=\BB_1\sqcup \BB_2$, where $\BB_1\not=\emptyset$ and $\BB_2\not=\emptyset$ 
cannot be connected with $\xi$-essential arcs in $\De$. Since there are at most $3(d+1)$ 
arcs (vertical and horizontal) in $\De$, and the total weight eminated from $\BB_1$ is at least $\si(d-1)$, 
one of the arcs eminated from $\BB_1$ has weight  at least $\xi (d)$. By the assumption, this arc must be vertical, 
so $\WW^\perp(\gamma) \geq  \xi$ --  contradiction. 

\msk  

We let now 
\begin{equation}\label{si(d)}
  \si(d) =  2^{-(d+2)}\si(d-1) < \xi(d)/2. 
\end{equation}
Assume 
\begin{equation}\label{bound on the flux}
   \WW^\perp(\gamma)< \si (d). 
\end{equation}
Then by the Claim the graph of the $\xi$-essential arcs in $\De$ is connected. 

Select some little Julia set $K_n$ surrounded by $\gamma$, 
and let $\om := (f^p)^\# \gamma$  be the pullback of $\gamma$ based at $K_n$ (immersed into $U\sm \KK$).    
Since $\deg(f^p: \om\ra \gamma) \leq 2^d$, we have  the usual estimate
$$
    \WW^\perp(\om) \leq 2^d \WW^\perp(\gamma) < \si (d-1).
$$
By the induction assumption, $\om$ cannot surround less than  $d$ little Julia sets, 
so it must suround all the $K_n$, $n\in \BB$ (and of course, nothing else). 

Moreover, $\om$ cannot cross any of the $\xi$-essintial arcs in $\De$, for otherwise
$\WW^\perp(\om) \geq \xi$ contradicting (\ref{si(d)}) and (\ref{bound on the flux}). 
Since the graph of these arcs is connected,
$\om$ is homotopic to $\gamma$ in  $U\sm \KK$. But then $f$ is renormalizable with intermediate period
$p/d \in (1,  p) $.  
\end{proof}


\begin{lem}\label{ess chains}
There exist $\xi= \xi(d)$ and  $C=C(d)$ such that for any $m$ with $|m|< N-l $ we have the following property:

\ssk \nin {\rm (i)} Either there is a vertical arc $\alpha$ emanating from $\De_m$  with $\WW(\alpha)> \xi$, 

\ssk \nin {\rm (ii)} Or else,   any $K_n \in \BB_m$  can be connected to 
each  neighboring block $\BB_{m+i}$, $i\in \{\pm 1\}$,
by a path of length $\leq C$ of $\xi$-essential arcs confined to 
$$
         \bigcup_{ |n-m|\leq C} \De_n. 
$$  
\end{lem}

\begin{proof}
Taking $\xi < \si(d)/3(d+1) $, we ensure that there an arc $\alpha$ eminated from $\De_m$ with $\WW(\alpha)> \xi$. 
If this arc is vertical then option (i) holds. Otherwise each $\De_m$ can be connected to a neigboring block, 
$\BB_{m+1}$ or $\BB_{m-1}$, with an essential arc.   

Let us orient the arcs $\alpha\in \AAA^\ess$ so that each $\alpha$ connects $\BB_m$ to $\BB_{m+i}$ with $i\geq 0$. 
  Let us push the $\BB_m$ and $\AAA^\ess$ to the cylinder $\Cyl$; 
we denote the corresponding vertices and arcs downstairs by $\bar K_n$ and  $\bar \alpha$. 
Let us consider the recurrent core $\bar \RR$ of $\bar \BB$ consisting of the vertices $\bar K_n$  that can be included in a loop
of $\bar\AAA^\ess$. Restriction of $\bar \AAA^\ess$ to $\bar\RR$ is ``connected with orientation'': any vertex can be connected to any other
by an orinented path of length bounded in terms of  $d$.
Hence for any $K_n$ and $K_s$ in $\RR_m\subset \BB_m$ (the lift of
$\bar \RR$ to $\BB_m$)
\note{connectivity issue}
there exists $i\geq 0$ (bounded in terms of $d$) such that 
 $K_n$ can be connected to $K_{s+i}$ in $\AAA^\ess$. 
Let us show that $i>0$ for some of these $i's$. Otherwise, the  connected component of $\RR_m$ (disregarding the orientation)
would be a finite isolated block in $\AAA^\ess$ whose size can be bounded {\it a priori } by some  $d'$ depending only on $d$.
This is impossible provided $\xi< \si(d')/ 3(d'+1)$.  

Since for any $K_t\in \RR_m$,  $K_{s+i}$ can be further connected  to $K_{t+j}$ with some (controlled) $j\geq i$, we see that 
 for any $K_n$ and $K_t$ in $\RR_m$ there exists $i > 0$ (bounded in terms of $d$) such that 
 $K_n$ can be connected to $K_{t+i}$ in $\AAA^\ess$.  Reversing orientation, we obtain the same property with $i<0$. 

Finally, if $\bar K_n\not \in \bar \RR$ then it can be connected to some $\bar K_s\in \bar \RR$ 
by a (disoriented) path of $\bar \AAA^\ess$ of length bounded in terms of $d$. The conclusion follows. 
\end{proof}

******* } 

\bignote{Section ``Loosing Julia sets'' is hidden below}

\comm{****
\subsection{Loosing Julia sets}

Let us say that a set $K_m$, $m\in \II_l$, {\it gets lost} under the $n$-th pullback
if $K_{m-n}\not\subset (f^n)^* \Om_l$. Note that then $K_{m+k}$ gets lost under any further pullback $(f^k)^*$, $k\geq n$. 

\begin{lem}\label{getting lost}
Assume that for each  set $K_l$, $l\in \LL$,  the following property holds:
There exists a vertex $K_j$  of the graph $\AAA|\, \II_l$ on distance
at most $\kappa N$ from $K_l$, 
 which gets lost under $n$-th pullback,   $n \leq  3 p$. Then the Amplification Thereom is valid. 
\end{lem}
 
\begin{proof}
Our assumption implies that  $K_j$ gets lost under the $3p$-th pullback, thus 
\begin{equation}\label{pulled-off J-set}
   K_j \not\subset (f^{3p} )^*\Om , \quad \Om \equiv \Om_l.
\end{equation}    \note{regular $\Om$}
By Lemma \ref{ess chains}, $K_j$ can be connected to $K_l$ by a chain of essential arcs confined to 
\begin{equation}\label{cup BB+m}
   \bigsqcup_{|m| \leq  \kappa  N+C} \BB_m.
\end{equation}
Then there exists $K_s\subset  (f^{3p} )^*\Om $ and 
$K_t\not\subset (f^{3p} )^*\Om $, $|s|, |t|\leq \kappa N+C$,
connected by an essential arc $\alpha$.
But then $\alpha$ restricts to a vertical arc in  $ (\Bf^p)^*\BOM $ of definite weight that begins on $K_s$.
By Lemma \ref{def vert does it}, this implies  the Amplification Thereom.  
\end{proof}

\begin{rem}
  We see, in partucular, that assuming  the contrary to the
  Amplification Theorem, 
the pullback $(f^n)^*\Om$ is independent of the base Julia  set $K_l$. 
\end{rem}


\begin{lem}\label{interesting moments}
Assume that for each $l\in \LL$ there exists a pullback $(f^n)^* \Om_l\not\supset K_1$, $n\leq 3p$,   
containing a horizontal arc $\alpha$ with endpoints on distance at
most $\kappa N$ from  $K_{l-n}$
such that  both  lifts  $f^* \alpha$  are pulled off.
Then  the concluison of the  Amplification Theorem  holds. 
\end{lem}
 
\begin{proof}
Take some $l\in \LL$, and let $\Om \equiv \Om_l$. 

Let the arc $\alpha$ lands on $K_i$ and $K_j$. By Lemma \ref{getting lost}, if one of the  Julia sets,  
 $K_{i-1}$ or $K_{j-1}$, gets lost then the concluison of the Amplification Theorem  holds. 
So, assume both $K_{i-1}$ and $K_{j-1}$  are contained in $( f^{n+1})^*\, \Om$. Let $\beta= f^* \alpha$ be the lift of $\alpha$ that begins on $K_{i-1}$. 
If it pulls off then  it ends on $K_{j-1}'$. Hence $K_{j-1}'\subset ( f^{n+1})^*  \Om$. 

Since  $( f^n)^* \Om $ does not contain $K_1$, 
the map $f: ( f^{n+1})^*\, \Om \ra ( f^n)^*\,  \Om $ is univalent. 
It follows that $K_{j-1}\not\subset ( f^{n+1})^*\,  \Om $, so this little Julia set gets lost.
The concluison follows. 
\end{proof}

********************************} 

\bignote{\S ``Good fit'' is hidden below}

\comm{*****
  
\subsection{Good  fit}

We say that a horizontal arc $\alpha$ {\it fits into a domain}
$(f^n)^* \Om$,  
$$ 
     \alpha\hsubset (f^n)^* \Om,
$$ 
 if it is contained in it up to homotopy in $U\sm \KK$ rel $\KK$.  
 If $\alpha\hsubset (f^n)^* (\Om[\kappa]) $,
we say that it  {\em fits $\kappa$-well  into } $(f^n)^* \Om$.

\bignote{Make the terminology and notations consistent!} 

\bignote{For one eye, we have an {\em amazing} fit! (except for a boun
ded number of pullbacks)}

\begin{lem}\label{fitting of  arcs}
   If an arc $\alpha$ fits $\kappa$-well into some domain
   $(f^n)^*\Om$,  
then it fits $(1-\kappa)$-well into all domains
$(f^{n+j})^*\Om$ for  $j=0,1,\dots,  (1-2\kappa) N $.
\end{lem}

\begin{proof}
      We have: 
$$
          (f^j)^* (\Om[1-\kappa] ) \hsupset \Om [\kappa] , \quad
          0\leq j  \leq      (1-2\kappa)  N. 
$$
Hence 
$$
  (f^{n+j})^* ( \Om [1-\kappa])  \hsupset (f^n)^*(\Om[\kappa]   )
  \hsupset \alpha  ,\quad j=0,1,\dots, (1-2\kappa)N. 
$$ 
\end{proof}

For $\xi\in (0,1)$, 
we say that a horizontal arc $\alpha\in \AAA$ is $\xi$-{\it essential} if $\WW(\alpha)\geq \xi \WW_\loc$. 
We let $\AAA^\ess$ be the weighted arc diagram comprised of all essential arcs
(we will often let $\xi$ be implicit in the terminology and
notation; just  an  apriori selected small number).
              \note{quantify  better?} 

\bignote{This def was transferred from \S ``Essential arcs''.}

\begin{lem}\label{fitting of ess arcs}
  Assume the conclusion of the Amplification Theorem fails.
  Let $\alpha$ be an essential horizontal arc.   
  Then there exists  $l\in \LL$  such that $\alpha$ fits $(1-\kappa)$-well into
   $ (1-2 \kappa) |\II|$ domains $(f^m)^*\Om$, $m< p$,
\end{lem}

\begin{proof}
  There are  at least $ (1-2d\kappa) |\JJ|$ domains  $(f^m)^*(\Om[1-\kappa])$,
  $m< p$, containing one endpoint of $\alpha$. If $\alpha$ does not
  fit into such a  domain, it creates a definite veritcal arc in it.
  If this happens for every $l\in \LL$ then
  by Lemma \ref{def vert does it},   the conclusion of the
  Amplification Theorem is valid.
\end{proof}

**************************}

\subsection{Confinement}

\note{ Move next to Lemma 2.2 ?} 

\begin{lem}\label{localization of local weight}
  Under the Translation Assumption, either the Amplification Theorem holds, or else 
  for any little Julia set $K_n$, there exists  $O(N)$  horizontal arcs $\alpha^n_i$ that begin on $K_n$ 
with total weight
$$
    \sum_i \WW(\alpha^n_i) \geq (1-\eps) \WW_n, 
$$
where $\eps=\eps(N)= O(\sqrt{1/N})$. Moreover, these arcs are confined to the pullback $(\Bf^j)^* \BOM$
of some translation domain $\BOM$
(here $n+j\equiv l \ \mod p$ and $0\leq j< p$). 
\end{lem}

\begin{proof}
Let us consider $L$ translation domains $\BOM_l$, $l\in \LL$. Assume that the vertical weight in each of these domains
emanating from the non-peripheral part $\II_l^\circ$ is at least $\WW_\loc/N$.
Then by  Corollary \ref{loc weights comparison} (ii)  
the vertical weight in 
$U\sm \LL$ emanating from each $\II_l^\circ$ and supported on vertical
arcs is at least $ \WW_\loc/ {2N}$, 
as long as $\WW_\loc$ is sufficiently big (in terms of $N$). 
Thus, if $L > 2 N \la$, we achieve amplification. 

\bignote{The $\Om_l$ should be selected well inside the translation region!} 

So, let us assume that for some translation domain $\BOM\equiv \BOM_l$ we have
$$
  \WW^\ver(\II^\circ|\, \BOM) \leq \frac  {\WW_\loc}N.  
$$
Together with property \ref{T4} this implies:  
$$
   \WW(\BOM\sm K_l) \leq \frac CN \, \WW_\loc. 
$$
By the Covering Lemma,
$$
     \WW( (\Bf^j)^* \BOM  \sm K_n ) \leq  \sqrt{4\,  \WW_\loc \cdot
       \WW(\BOM\sm K_l )} \leq \frac C {\sqrt{N}}\,  \WW_\loc, 
$$
with a constant $C$ depending only on $d$.
Hence the rest of the local weight at $K_n$ is the horizontal weight confined to  $(\Bf^j)^* \BOM$. Since this domain contains 
$O(N)$ little Julia sets, this horizontal weight, up to an additive constant $O(N)$,  is localized on at most $O(N)$ horizontal arcs.
\end{proof}

\subsection{Proof of the Amplification Theorem. }
\comm{******
Assume the conclusion of the Amplification Theorem fails.  
If $L$ is big enough, then by Lemmas \ref{getting lost} and \ref{interesting moments},
for some $l\in \LL$ no little Julia set and no horizontal arc (of the type specified in the lemmas) 
get lost (respectively: pulled off) under $3p$ pullbacks of
$\BOM\equiv \BOM_l$;
and moreover, the fitting property of Lemma \ref{fitting of ess arcs} is also satisfied
for this $\Om$.  
Then ******* }
Assuming that  the conclusion of the Amplification Theorem fails,  
we will show that one of the domains $(\Bf^n)^*\, \BOM_l$, $ n\leq 3p$, contains a definite vertical weight 
(i.e., of order $N\cdot \WW_\loc$), and the Covering Lemma will complete the proof. 



By Lemma \ref{sym vert arcs} (adapted to our situation), the domain 
$V:= ( f^{l+p})^*\, \Om \supset K_0$ \note{!} 
contains two symmetric vertical curves
landing on $K_0$ and disjoint from $\AAA^\hor(V)$. 

Let $\AAA_l$ be the set of horizontal  arcs that begin on $K_l$.  
\note{Defined in \S \ref{assumptions sec}?} 
Then Lemma \ref{in p steps} ensures that all $\alpha\in \AAA_l$
get pulled off in less than $p$ iterates.  
Let $n=n(\alpha)$ be the pull-off moment, i.e.,
there exist a legitimate horizontal pullback $(f^k)^*\alpha $  for $k=0,1,\dots , n-1$, 
while  $(f^n)^* \alpha$ pulls off. 
By  Lemma  \ref{interesting moments},
 $n(\alpha) \in \II_l$. \note{Need subscript $l$?}

 Given an $\eps>0$, let $\xi=\eps/3d$. 
Let $\AAA_l^\ess$ be the set of $\xi$-essential arcs that begin on
$K_l$.  Their total weight is at least $(1-\eps)\WW_l$ . 

\bignote{Next line: These are the moments when $(f^n)^*\BOM \supset K_1$ !} 

Let $\nu_n$ be the total horizontal  weight in $(\Bf^{n-1})^*\, \BOM'$ that gets pulled off under $\Bf^*$, $n\in \II$,
and let $\eta_n$ be the part of $\WW_0$ that was used at that moment  for creation
of new horizontal arcs. 
Then the new horizontal weight that can be created in  $(\Bf^n)^*\, \BOM' $ is bounded by 
\begin{equation}\label{bound on nu}
    \frac 14 (2\nu_n  + \eta_n) +O(N). 
\end{equation}
Hence the loss of the horizontal weight as we pass from   $(\Bf^{n-1})^*\, \BOM'$ to  $(\Bf^n)^*\, \BOM'$ is at least
\begin{equation}\label{loss at one step}
    \nu_n  -  \frac 14 (2\nu_n +  \eta_n ) - O(N) = \frac 12 \nu_n - \frac 14 \eta_n  - O(N). 
  \end{equation}\label{nu-prime}

  Let us represent $\eta_n$ as $\sum \eta(\alpha)$
 over the horizontal arcs $\alpha\in (f^{n-1})^* \BOM'  $ that gets
pulled off under $f^*$,
where $\eta (\alpha)$ is the part of $\WW_0$ used for creation of
new horizontal arcs connecting $K_0$ to $f^*\alpha$.
Then we have:
\begin{equation}\label{penalty} 
  \sum \eta_n  = \sum_{n\in \II}  \sum_{\alpha\in (f^{n-1})^*\BOM'}  \eta (\alpha) =
  \sum_\alpha\sum_n  \eta (\alpha) \leq  (2N+1) \sum_\alpha\eta(\alpha) 
  \leq (2N+1) \WW_1. 
\end{equation}  

On the other hand, let us now show that
\begin{equation}\label{bounds on sums}
    \sum_{n \in \II_l} \nu_n  \geq \frac {  2N+1  }2 \WW_0 (1-\eps).
\end{equation}
\comm{****  since:

\ssk\nin (i)
 All horizontal arcs of $\AAA_l $ get pulled off at moments $n\in\II$;

\ssk\nin (ii)
The total weight of these arcs is at least $\WW_l (1-\eps)$;

\ssk\nin (iii)
Each of the arcs $(f^{  n(\alpha) -1  })^* \alpha$, where $\alpha\in \AAA_l^\ess$,
fits well  into $(1-2 \kappa)\,  |\II|$ domains
$(f^{m-1})^* \Om$ (Lemma \ref{fitting of ess  arcs});

\ssk\nin (iv)
Each of these arcs has two ends: this yields  $1/2$;

\ssk\nin (v)
  At the moments $n\not\in \II$,  we do not gain any  horizontal weight.   

\bignote{More details !} 

*********} 
Let $K_m$, $m\in \II$,  be the family  of little Julia sets in $V$
(introduced in the beginning of the proof) 
and   let $\AAA_0$ be the set of horizontal arcs in $\BV$  that begin on
$K_0$ and fit into $ V$.  
 By Lemma  \ref{interesting moments},
 for any  horizontal arc $\beta \in \AAA_0$,  
the  pulled-off  moment $n(\beta) $ belongs to $\II$. 
Since the correspondence $\beta\mapsto (f^{n(\beta)-1} )^*\beta$
is at most  two-to-one, the total weight of the arcs  
$(f^{n(\beta)-1} )^*\beta$, $\beta\in \AAA_0$,  is at least $\WW_0/2$. 
Moreover,
each of the arcs $(f^{n(\beta)-1} )^*\beta$ fits into one of the $2N+1$ \note{explain}
domains $ (f^{m-1})^* \Om ' $, $m \in \II_0$,   \note{ $\Om'= (f^p)^* (\Om)$ } 
implying (\ref{bounds on sums}).

\bignote{In fact
  all the events in question happen within a bounded range of moments.
  Indeed, the loss of the short arc must happen at the Julia
  set $K_i$, $-\bb\leq i <0$  (before the short arc enters the
  translation region). At this moment one can form some arcs from $K_0$
  to $K_i$. These arcs must be lost at some $K_i$, $i\geq -2\bb$, etc.
  Hence the essential weight emanating from $K_0$ must be formed in bounded
  number of steps. }

Summing (\ref{loss at one step} ) up over $n\in \II$, taking into
account  (\ref{penalty}) and  (\ref{bounds on sums}), 
we see that the total horizontal loss is at least
\begin{equation}\label{loss of ess weight}
   \frac {2N+1}  4 [  (1-\eps)\WW_0 -\WW_1] - O(N^2 ).
\end{equation}
So, if $\WW_0$ is definitely bigger than $\WW_1$, we are done. 

\comm{*******
\msk
Assume this is not the case.
Assume also that  $\WW_0$ is definitely bigger than
$\WW_1$ (and hence is definitely bigger than $\WW_l$  as well):
$$
     \WW_0- \WW_l \geq 2\de \WW_\loc,
$$
with an absolute $\de>0$. 

By Lemma \ref{localization of local weight}, \note{!!} 
the total weight of  $\AAA_0$
is at least $(1-\eps)\WW_0$, where $\eps = O(1/\sqrt{\kappa N})$, so
it is definitely bigger than $\WW_l$.  Let $\tl \AAA_0$ be the family of those arcs
$\beta\in \AAA_0$ that are not  $f^l$-pullbacks of arcs $\alpha\in \AAA_l $.  
\bignote{It can be old arcs with a bigger weight -- correct counting (?) }
The
total weight of these arcs is at least 
$$
  (1-\eps)\WW_0 - \WW_l\geq \de \WW_\loc
$$ 
(if $\eps$ is small enough).
Each arc $\beta$ gets pulled off under $(f^n)^*$ at some moment $n(\beta)<p$.
In fact, $n(\beta)\leq p - \kappa \qq$ (if $\beta$ is essential),
for otherwise $(f^{n(\beta)-1} )^*\beta$
would be a short arc in the translation region (by Corollary \ref{loc weights comparison}),
which cannot be pulled off.
********}

\msk  
On the other hand, if $\WW_0$ is approximately equal to $\WW_1$, 
we replace estimate (\ref{bound on nu}) with 
$$
      (\WW_0-\WW_1) +  \nu_n/ 2 = \eps' \WW_\loc +\nu_n/2,
$$
which results in the  horizontal loss $\nu_n/2- \eps' \WW_\loc$.
Summing these up, taking (\ref{bounds on sums}) into account,
we produce, once again,  a total horizontal loss of order $N\WW_\loc$. 



\section{Local connectivity}

Our {\em a priori} bounds have an extra virtue of being {\em unbranched}:
the fundamental annuli with definite moduli do not intersect any non-central little
Julia sets (of the top level). 
Local connetivity of the Julia sets under considertation now follow from the
result of Hu and Jiang \cite{HJ,J}
(see also \cite{McM} and \cite[\S 9.2] {puzzle}).
MLC at the corresponding parameter values follows from the machinery developed in
\cite{puzzle} and existence of little Mandelbrot copies on
Ecall\'e-Voronin cylinders, see \cite{C,DLP}
(making use of the theory of geometric limits for parabolic bifurcations \cite{D,La}). 
%

Let us outline the MLC argument  in more detail.
To prove MLC, we need to show that two topologically conjugate maps $f$ and $\tl f$
in question
are qc equivalent.  To this end, it is enough to construct a qc map
which is homotopic to the conjugacy rel postcritical set (``qc Thurston equivalence'').
Here are the main steps of the construction:

\ssk {\em Step 0.}
Start with a conjugacy $h_0$ between $f$ and $\tl f$
which is $K_0$-qc outside the Julia sets,
 with $K_0$ depending only on {\em a priori} bounds,
 
\ssk {\em Step 1.}
Adjust it to a $K$-qc  map $h$, with $K$ depending only on
{\em a priori} bounds, which coincides with the standard conjugacy on the
configuration of $\alpha$-rays. It can be done making use of holomorphic motions
of the ray configurations over neighborhoods of little Mandelbrot copies.
To have a uniform dilatation control, we need to see llittle Mandelbrot copies
on the Ecall\'e- Voronin cyclinders, see \cite{C,DLP}.%
\footnote{ (added in Jan 2026) The project \cite{DLP} was further developed
jointly with Alex Kapiamba,
to have a result which is exactly suitable for our needs (to appear).}

\ssk {\em Step 2.} We then use Yoccoz puzzle to construct a $K'$-qc map
$h' : \C\ra \C$ (with $K'$ depending only on {\em a priori } bounds but slightly worse than
$K_0$) in the right homotopy class which is equivariant of the little
Julia sets of top level
(see \cite[\S 11]{puzzle}).

\ssk {\em Step 3.} We adjust $h'$ in exterior neighborhoods of the little Julia sets so that it acquires the initial 
delatation $K_0$ over there  (see \cite[\S 7.5] {puzzle}), producing a new Thurston map $h_1$.
It can be used to start the construction over for renormalizations near little Julia sets
without any loss of dilatation.

\end{document}